\documentclass[12pt]{amsart}
\usepackage{amsmath}
\usepackage{amsthm}
\usepackage{amssymb}
\usepackage{amscd}
\usepackage{amsfonts}
\usepackage{amsbsy}
\usepackage{epsfig,afterpage}
\usepackage[dvips]{psfrag}
\usepackage{subfigure}
\usepackage{epsfig}
\usepackage{amsmath}
\usepackage{srcltx}
\usepackage{amsfonts}
\usepackage{amssymb}
\usepackage{psfrag}
\usepackage{verbatim}
\usepackage{graphicx}
\usepackage{amsmath}
\usepackage{amsthm}
\usepackage{amssymb}
\usepackage{amscd}
\usepackage{amsfonts}
\usepackage{amsbsy}
\usepackage{epsfig}
\usepackage[dvips]{psfrag}
\usepackage{multirow}
\usepackage{graphics}
\usepackage{epstopdf}
\usepackage{overpic}
\usepackage{float}

\usepackage{graphicx}

\usepackage[colorlinks=true, linkcolor=blue, citecolor=red]{hyperref}

\def\e{\varepsilon}

\newtheorem {theorem} {Theorem}
\newtheorem {proposition} [theorem]{Proposition}
\newtheorem {corollary} [theorem]{Corollary}

\newtheorem {example} [theorem]{Example}
\newtheorem {remark} [theorem]{Remark}

\newtheorem{mtheorem}{Theorem}

\usepackage{float}
\usepackage{tikz, wrapfig}
\tikzset{node distance=3cm, auto}
\allowdisplaybreaks

\begin{document}

\title[Holomorphic slow-Fast systems]
{ Holomorphic slow-Fast systems}

\author[Gabriel Rondón, Paulo R. da Silva and Luiz F. S. Gouveia]
{Gabriel Rondón, Paulo R. da Silva and Luiz F. S. Gouveia}

\address{S\~{a}o Paulo State University (Unesp), Institute of Biosciences, Humanities and
	Exact Sciences. Rua C. Colombo, 2265, CEP 15054--000. S. J. Rio Preto, S\~ao Paulo,
	Brazil.}

\email{gabriel.rondon@unesp.br}
\email{paulo.r.silva@unesp.br}
\email{fernando.gouveia@unesp.br}

\thanks{ .}


\date{}
\maketitle

\begin{abstract}
In this paper, we are concerned with studying the existence of invariant complex manifolds of two-dimensional holomorphic systems. From the geometric singular perturbation theory we know that if a slow-fast system has associated a normally hyperbolic compact critical manifold, then there exists a smooth locally invariant manifold. However, this smooth manifold does not necessarily have a complex structure. Here, we provide conditions to guarantee the existence of one-dimensional invariant complex manifolds. Consequently, this allows us to establish that the centers, foci, and nodes of the reduced problem are persistent by singular perturbation. The tools used by us are the usual techniques of Fenichel and Briot-Bouquet Theories. 

\end{abstract}


\section{Introduction}
In the qualitative theory of dynamical systems, there are several studies on singular perturbation problems (see, for instance, \cite{Fenichel79,Jones95,reg_space,TexeiraSilva12}). In this paper, we want to study singular perturbation problems in $\mathbb{R}^4,$ which can be written as two-dimensional holomorphic systems. It is important to highlight that although there are applications that can be modeled by two-dimensional holomorphic systems, see \cite{Moroz_b,Moroz_a}, there are few articles on the subject such as \cite{Needham1,Needham2}.



The basic systems we consider are of the form 
\begin{equation}\tag{SF1}\label{main_eq}
\left\{\begin{array}{rcl}
\varepsilon\dot{z}&=&f(z,w),\\\
\dot{w}&=&g(z,w),
\end{array}\right.
\end{equation}
where $f,g:D\to\mathbb{C}$ are holomorphic functions,  $D\subseteq\mathbb{C}^2$ is a simply connected domain, $\varepsilon>0$ is a small enough parameter and the dot $\cdot$ represents the derivative of the functions $z(\tau)$ and $w(\tau)$ with respect to the real variable $\tau$.

Emphasize that $C_0=\{(z,w)\in D:f(z,w)=0\}$ is the one-dimensional critical manifold associated with the system \eqref{main_eq} provided that $\frac{\partial f}{\partial z}(z,w)\neq 0$ for all $(z,w)\in C_0$. Recall that \eqref{main_eq} can be written as a real four-dimensional autonomous system in an appropriate domain of $\mathbb{R}^4$.

Suppose that $S_0\subset C_0$ is a normally hyperbolic compact critical manifold. Thus, we can assume that $S_0$ is given as the graph of a function of $z$ in terms of $w$. Indeed, since $\frac{\partial f}{\partial z}(z,w)\neq 0$ for all $(z,w)\in S_0,$ then from \textit{Implicit Function Theorem} (see \cite{volker}), there exist an open set $U$ and a holomorphic function $z=h(w)$ such that 
$f(h(w),w)=0,$ for all $w\in U.$ Hence, without loss of generality we can assume that $S_0=\{(z,w):z=h(w),w\in U\}.$ 

One of the main goals of this paper is to study the existence of invariant complex manifolds of two-dimensional holomorphic system \eqref{main_eq}. At some point, the reader may think that it is enough to apply \textit{Fenichel's theorem} (see, for instance, \cite{Jones95}) to obtain a one-dimensional invariant complex manifold. In general, this is not true, due to the fact that there are smooth manifolds that have no complex structure. An obvious impediment is that the dimension of the smooth manifold must be even. Another obstacle is orientability, since every complex manifold is orientable.

Although the \textit{geometric singular perturbation theory} does not allow us to guarantee the existence of invariant complex manifolds, it is possible to use this theory and the Laurent series to construct examples of smooth manifolds without complex structure. Indeed, if we apply the \textit{Fenichel's Theorem} to the equivalent $C^\infty$ four-dimensional real system associated with system \eqref{main_eq}, then there exists a smooth locally invariant manifold $S_\varepsilon=\{(z,w):z=h_\varepsilon(w)\}$ of slow-fast system \eqref{main_eq}, which is diffeomorphic to $S_0.$ In general, this function $h_\varepsilon(z)$ does not have to be holomorphic, as we illustrate in the following example. 
Consider the complex dynamical system
\begin{equation}\label{main_eq_ex}
\begin{aligned}
\left\{\begin{array}{l}
\varepsilon\dot{z}=z+w,\\\
\dot{w}=w^2,
\end{array} \right.
\end{aligned}
\end{equation} 
notice that $C_0=\{(z,w):z=-w\}$ is a normally hyperbolic critical manifold of \eqref{main_eq_ex}. In addition, if  $S_0\subset C_0$ is a compact set, then the equivalent system to \eqref{main_eq_ex} in $\mathbb{R}^4$ has a smooth locally invariant manifold $S_\varepsilon$. Nevertheless, this manifold has no complex structure. Indeed, suppose that there exists a holomorphic function $h_\varepsilon$ such that $S_\varepsilon=\{(z,w):z=h_\varepsilon(w)\}$. Then $h_\varepsilon$ can be written as a power series. In Section \ref{sec:Fenichelaprox}, we will prove that this series is given by
$$h_\varepsilon(w)=-\sum_{k=0}^{\infty} \varepsilon^k k!w^{k+1}.$$ Notice that this series diverges for $w\neq 0$, which contradicts the fact that $h_\varepsilon$ is holomorphic.

At this point, it is natural to ask: under what conditions can we ensure the existence of invariant manifolds with complex structure of system \eqref{main_eq}?
In Section \ref{sec:BB}, we use the \textit{Briot-Bouquet Theory} and provide conditions to guarantee the existence of one-dimensional invariant complex manifolds. Specifically, in Theorem \ref{main_teo} we prove that, for $\varepsilon>0$ sufficiently enough, the  system 
\begin{equation}\tag{SF2}\label{per_sys_eq_nf}
\begin{aligned}
\left\{\begin{array}{l}
\varepsilon\dot{z}=\alpha z+\tilde{f}(z,w),\\[5pt]
\dot{w}=\beta w+\tilde{g}(z,w),
\end{array} \right.
\end{aligned} \,\text{where}\, \tilde{f},\tilde{g}=\mathcal{O}_2(z,w) \,\text{and}\, \alpha,\beta\in\mathbb{C},
\end{equation}
 has a unique one-dimensional invariant complex manifold $C_\varepsilon$ passing through the equilibrium point $q_\varepsilon=(0,0)$ provided that the critical manifold is normally hyperbolic. 

We employ Theorem \ref{main_teo} to prove that the centers, foci and nodes of the reduced problem associated with the holomorphic slow-fast system \eqref{main_eq}, whose Jacobian matrix of $f$ and $g$ is diagonalizable at the origin, are preserved by singular perturbation, for more details see Theorem \ref{prop_center_2} and Figure \ref{fig-foci}.
\begin{figure}[h!]
\begin{overpic}[scale=0.26]{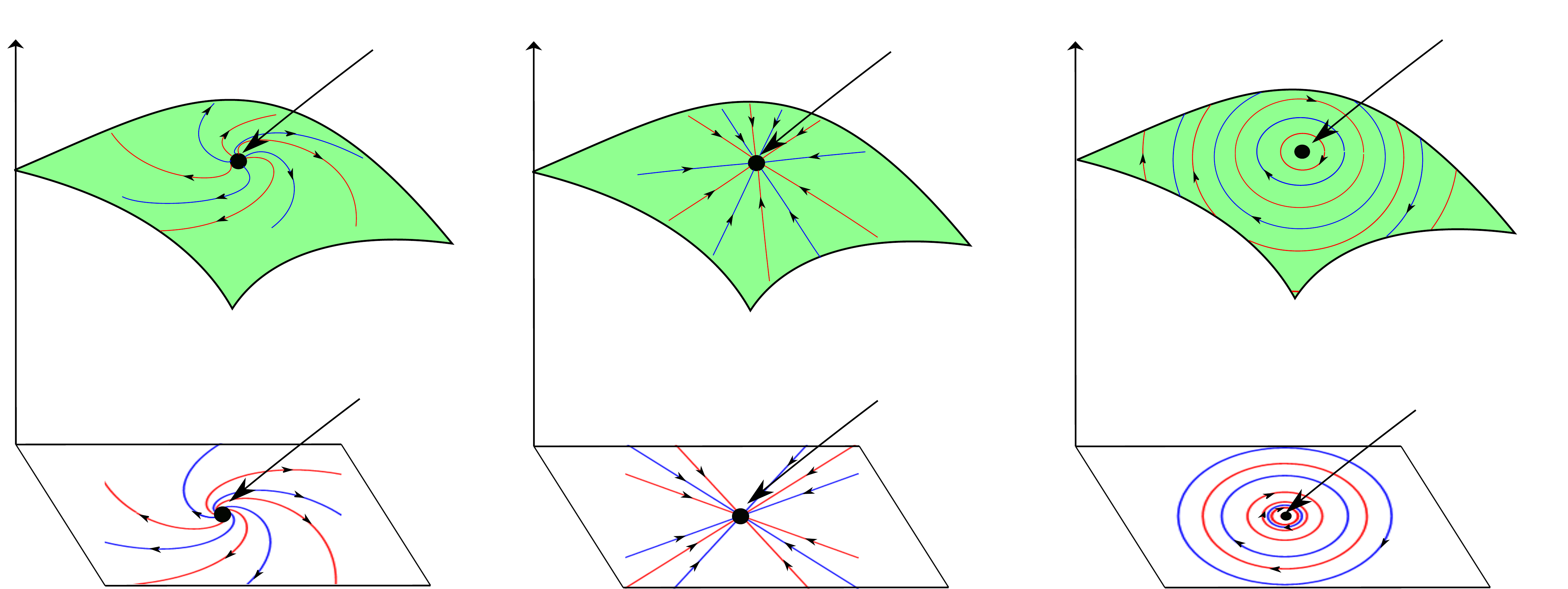}
		\put(96,1){$C_0$}
        \put(61,1){$C_0$}
        \put(28,1){$C_0$}
        \put(97,24){$C_\varepsilon$}
        \put(61,24){$C_\varepsilon$}
        \put(28,24){$C_\varepsilon$}
		\put(24,13){$q_0$}
        \put(57,13){$q_0$}
        \put(91,12){$q_0$}
        \put(25,35){$q_\varepsilon$}
        \put(58,35){$q_\varepsilon$}
        \put(93,35){$q_\varepsilon$}
        \put(0.5,36){$\varepsilon$}
        \put(33.5,36){$\varepsilon$}
        \put(68,36){$\varepsilon$}
		\end{overpic}
  \caption{\footnotesize{Persistence of a focus, node and center of the reduced problem by singular perturbation, respectively.}}
  \label{fig-foci}
\end{figure}
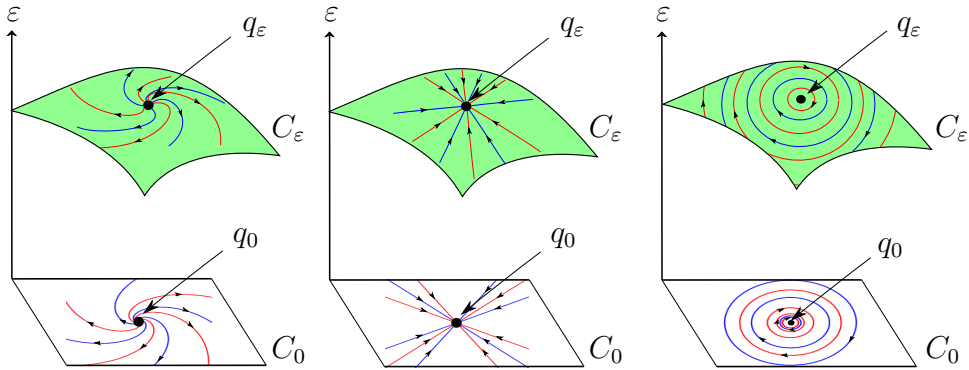

Also, we are concerned with determining the existence of complex invariant manifolds for certain families of two-dimensional holomorphic systems, whose jacobian matrix of $f$ and $g$ is not necessarily digonalizable at the origin and this point does not need to be an equilibrium point of the system, see Theorems \ref{uncouple1} and \ref{couple}. Specifically, Theorem \ref{uncouple1} stablishes the existence of one-dimensional invariant complex manifolds associated with slow-fast system \eqref{main_eq} whose equation of orbits is separable, that is, in the case that there exist holomorphic functions $\eta,\kappa$ such that $\frac{g(z,w)}{f(z,w)}=\eta(z)\kappa(w)$ in a suitable domain. 
Furthermore, in Theorem \ref{couple} we state conditions to ensure when the system 
\begin{equation}\tag{SF3}\label{couple_eq_1}
\begin{aligned}
\left\{\begin{array}{l}
\varepsilon\dot{z}=\alpha z+\beta G(w),\\[5pt]
\dot{w}=g(w),
\end{array} \right.
\end{aligned} \text{where} \ G'\equiv 1/g \, \text{and}\ \alpha,\beta\in\mathbb{C}\setminus\{0\},
\end{equation}
has associated a one-dimensional exponentially attractive invariant complex manifold $C_\varepsilon$.

Finally, we will present a method to study slow-fast systems of the form
\begin{equation}\tag{SF4}\label{main_eq_aprox}
\begin{aligned}
\left\{\begin{array}{l}
\varepsilon\dot{z}=\alpha z+\beta w,\\[5pt]
\dot{w}=g(w),
\end{array} \right.
\end{aligned}, \,\text{where}\, \alpha\in\mathbb{C}\setminus\{0\}\, \text{and}\, \beta\in\mathbb{C},
\end{equation}
where $g$ is one of the normal forms given in \cite{GGJ}, namely: $\dot{w}=1,$ $\dot{w}=(a+ib)w$, $\dot{w}=w^n$, $\dot{w}=\frac{1}{w^n}$ and $\dot{w}=\frac{\gamma w^n}{1+w^{n-1}}$. We will determine which of these families could have associated manifolds with complex structure and in that case we will investigate what happens when a singularity of the reduced problem is perturbed for $\varepsilon>0.$ Specifically, in Propositions \ref{prop:1} and \ref{prop:2} we prove that the centers and poles of order $n$ of the reduced problem associated with system \eqref{main_eq_aprox} are persistent by singular perturbation.


The paper is organized in the following form. In Section \ref{sec:Preliminaries}, we present some basic results on the complex manifolds, geometric singular perturbation theory and the Briot-Bouquet Theory. In Section \ref{sec:BB}, we employ Briot-Bouquet theory to prove Theorem \ref{main_teo} and \ref{prop_center_2}. In Section \ref{sec:mainresults}, we state and prove Theorems \ref{uncouple1} and \ref{couple}. Lastly, in Section \ref{sec:Fenichelaprox} we use the Laurent series to give complex systems that have no associated complex manifolds. 

\section{Preliminaries}\label{sec:Preliminaries}
This section is devoted to establishing some basic results that will be used throughout the paper.
\subsection{Complex manifolds}

 A  \textit{holomorphic function}  $f$
is a complex-valued function  defined in a domain $\mathcal{V}\subseteq\mathbb{C}$ and satisfying that
\begin{itemize}
	\item $u=\operatorname{Re}(f)$ and $v=\operatorname{Im}(f)$  are continuous,
	\item  there exist the partial derivatives $u_x,u_y,v_x,v_y$ in $\mathcal{V}$ and, 
	\item  the partial derivatives satisfy the Cauchy--Riemann equations, see \cite{Morris}, \[u_x=v_y,\quad u_y=-v_x,\quad \forall z=x+iy\in\mathcal{V}.\]  
\end{itemize}
We recall that \textit{Looman--Menchoff's Theorem} establishes that the above conditions are sufficient to guarantee the analyticity of  $f$. This result implies that for any $z_0\in\mathcal{V}$ 
\begin{equation*}
f(z)=A_0+A_1(z-z_0)+A_2(z-z_0)^2+...,\quad A_k=a_k+ib_k=\dfrac{f^{(k)}(z_0)}{k!}
\label{analF}
\end{equation*}
for $z\in D(z_0,R_{z_0})\subseteq\mathcal{V}$ where $D(z_0,R_{z_0})$ is the largest possible $z_0$--centered disk contained in $\mathcal{V}.$ 

If $f$ is holomorphic in a punctured disc $D(0,R)\setminus\{0\}$ and it is not derivable at $0$ we say that $0$ is a singularity of $f$. In this case $f$ is equal to its Laurent's series in $D(0,R)\setminus\{0\}$

\begin{equation*} 
	f(z)=\sum_{k=1}^{\infty}\dfrac{B_k}{z^k}+ \sum_{k=0}^{\infty}A_kz^k, \label{laurent}
\end{equation*}
where $B_k=\frac{1}{2\pi i}\int_{C_{\e}}f(z)z^{k-1}dz,\quad A_k=\frac{1}{2\pi i}\int_{C_{\e}}\frac{f(z)}{z^{k+1}}dz$
with $C_{\e}$ parameterized by $z(t)=\e e^{it}, \e\sim0$, $t\in[0,2\pi]$.  

If $B_k\neq0$  for an infinite set of indices $k$ we say that $0$ is an \textit{essential singularity} and if
there exists $n \geq1$ such that $B_n \neq0$ and $B_k = 0$ for every $k>n$ then we say that $0$ is a \textit{pole of order n}. 

The notion of holomorphic function can be extended as follow.
\\

    Consider $\mathcal{W}\subset\mathbb{C}^2$ an open subset and let $f:\mathcal{W}\rightarrow\mathbb{C}$ be a continuously differentiable function. Then $f$ is said \textit{holomorphic} if the Cauchy-Riemann equations hold for all coordinates $z_j=x_j+iy_j,$ with $j=1,2,$ i.e. \[u_{x_j}=v_{y_j},\quad u_{y_j}=-v_{x_j},\quad \forall z_j=x_j+iy_j\in\mathcal{W}.\]  

In addition, if $f$ is holomorphic in $\mathcal{W},$ then in each point $\zeta=(\zeta_1,\zeta_2)\in \mathcal{W}$ has an open neighborhood $D(\zeta,R_\zeta)$, such that the function $f$ can be expanded into a power series
$$f(z,w)=\sum_{n=1}^\infty\sum_{k_1+k_2=n}c_{k_1,k_2}(z-\zeta_1)^{k_1}(w-\zeta_2)^{k_2},$$
which converges for all $(z,w)\in D(\zeta,R_\zeta)$.





The definition of a complex manifold is analogous to the smooth manifold, but we require that the charts take on values in $\mathbb{C}^2$ and that the transition functions be holomorphic (see Figure \ref{fig-transition}).
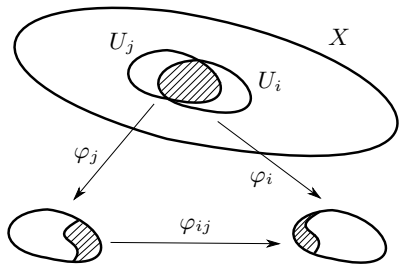
\begin{figure}[h!]
  		\begin{overpic}[scale=0.22]{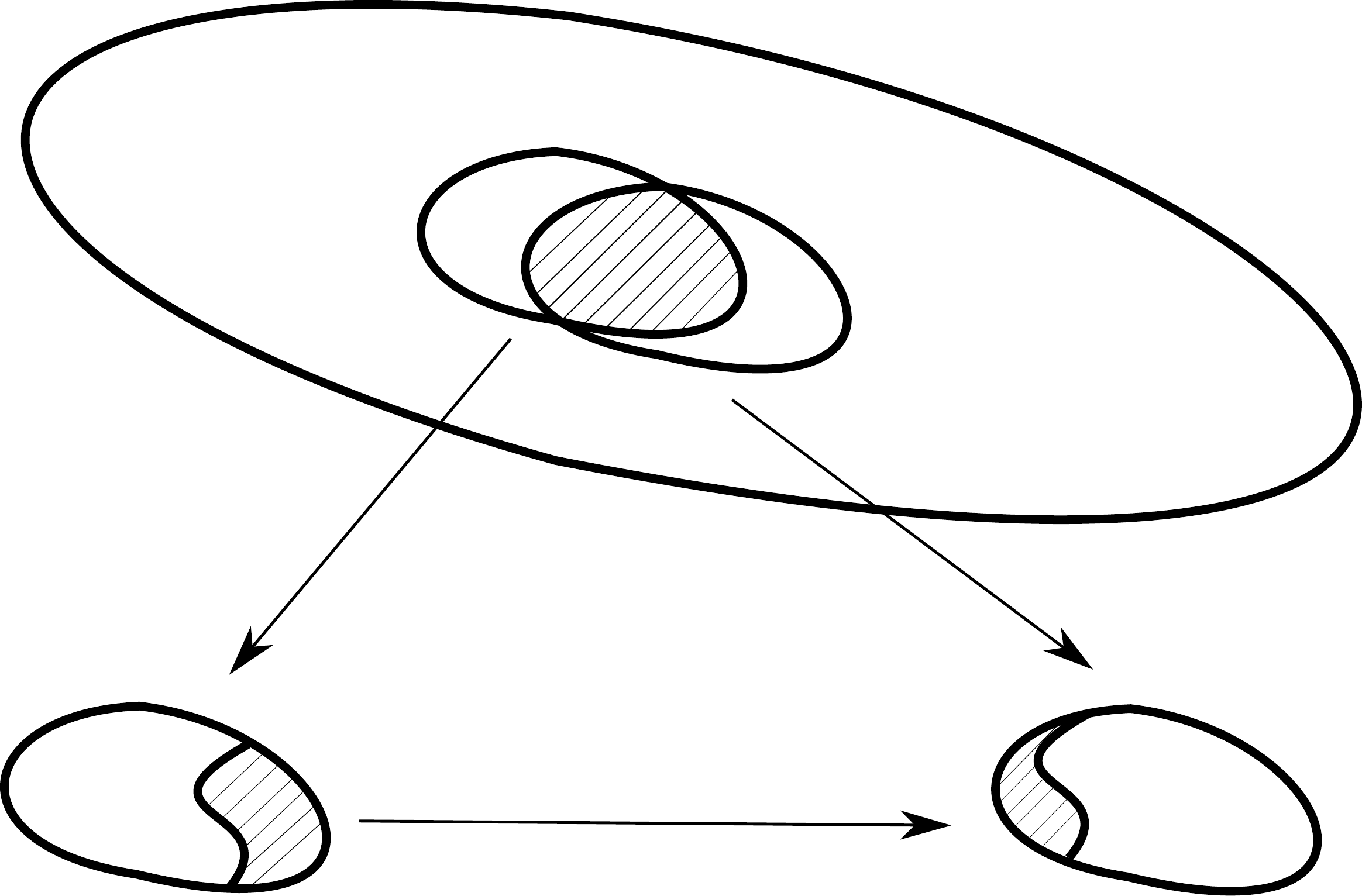}
		\put(26,55){{\scriptsize $U_j$ \par}}
		\put(64,45){{\scriptsize $U_i$ \par}}
        \put(82,56){{\scriptsize $X$ \par}}
		\put(17,27){{\scriptsize $\varphi_j$ \par}}
        \put(62,21){{\scriptsize $\varphi_i$ \par}}
        \put(44,9){{\scriptsize $\varphi_{ij}$ \par}}
		\end{overpic}
  \caption{\footnotesize{Complex manifolds require that the transition functions $\varphi_{ij}$ be holomorphic.}}
  \label{fig-transition}
\end{figure}

A \textit{holomorphic atlas} in a topological space $X\subseteq\mathbb{C}^2$ is a collection of pairs $(U_i, \varphi_i)$ called holomorphic charts where
\begin{itemize}
    \item Each $U_i$ is an open in $X$ and $X =\displaystyle\bigcup_{i}U_i;$
    \item Each $\varphi_i: U_i\rightarrow \mathbb{C}^n$ is a homomorphism over an open of $\mathbb{C}^n$, with $n\leq 2$;
    \item Whenever $U_i\cap U_j\neq\emptyset$ the transition 
    $$\varphi_{ij} = \varphi_i\circ \varphi^{-1}_j: \varphi_j(U_i\cap U_j)\rightarrow\varphi_i(U_i\cap U_j)$$
    is a holomorphic map.
\end{itemize}

A \textit{complex manifold} is a topological space $X$, Hausdorff and with enumerable base, equipped with a maximal holomorphic atlas. The number $n$ is called the complex dimension of $X$.

The following proposition will give us a way to obtain complex submanifolds. For more details see, for instance, \cite{Huybrechts,wellray}.

\begin{proposition}\label{submanifolds}
Let $\phi: X\subseteq\mathbb{C}^2\rightarrow \mathbb{C}$ be a holomorphic map between complex manifolds. Consider $b\in \phi(X)\subseteq \mathbb{C}$ such that the rank of $\phi$ is maximal ($rank(\phi) = 1$), for all $a\in \phi^{-1}(b)$. Then $\phi^{-1}(b)$ is a complex submanifold of X of dimension $1$.
\end{proposition}
\subsection{Geometric singular perturbation theory}
We consider singularly perturbed systems of differential equations
\begin{equation}\label{eq-def-slowfast-1}
\begin{aligned}
\left\{\begin{array}{l}
 \varepsilon\dot{\mathbf{x}}  =  f(\mathbf{x},\mathbf{y},\varepsilon), \\ 
    \dot{\mathbf{y}}  =  g(\mathbf{x},\mathbf{y},\varepsilon),
\end{array} \right.
\end{aligned}
\end{equation}
which is called \textit{slow-fast system}, where  $\varepsilon\in(0,\varepsilon_0),$ $\varepsilon_0$ small, $\mathbf{(x,y)}\in\mathcal{M}$, an open set $\mathcal{M}\subset\mathbb{R}^{4}$ and $f,g:(0,\varepsilon_0)\times \mathcal{M}\rightarrow \mathbb{R}^{4}$ are $C^r$. The dot $\cdot$ represents the derivative of the functions $\mathbf{x}(\tau)$ and $\mathbf{y}(\tau)$ with respect to the variable $\tau$.

If we write $t = \displaystyle\frac{\tau}{\varepsilon}$, then system \eqref{eq-def-slowfast-1} becomes
\begin{equation}\label{eq-def-slowfast-2}
\begin{aligned}
\left\{\begin{array}{l}
 \mathbf{x}' =  f(\mathbf{x},\mathbf{y},\varepsilon), \\ 
    \mathbf{y}' = \varepsilon g(\mathbf{x},\mathbf{y},\varepsilon),
\end{array} \right.
\end{aligned}
\end{equation}
in which the apostrophe ' denotes the derivative of the functions $\mathbf{x}(t)$ and $\mathbf{y}(t)$ with respect to the variable $t$. Notice that the parameter $\varepsilon = \displaystyle\frac{\tau}{t}$ represents the ratio of the time scales.

Consider equation \eqref{eq-def-slowfast-1} and set $\varepsilon = 0$. We obtain the so called \textit{reduced problem} given by
\begin{equation}\label{eq-def-slow-system}
    0 = f(\mathbf{x},\mathbf{y},0), \ \ \
    \dot{\mathbf{y}} = g(\mathbf{x},\mathbf{y},0).
\end{equation}

Observe that \eqref{eq-def-slow-system} is not an  ordinary differential equation, but it is an \textit{algebraic differential equation}.

Solutions of \eqref{eq-def-slow-system} are contained in the set
$$C_{0} = \Big{\{}(\mathbf{x},\mathbf{y})\in\mathcal{M}: \ f(\mathbf{x},\mathbf{y},0) = 0\Big{\}}.$$

The set $C_{0}$ is called \textit{critical set}. In the case where $C_{0}$ is a manifold, $C_{0}$ is called critical manifold.

On the other hand, setting $\varepsilon = 0$ in equation \eqref{eq-def-slowfast-2} we obtain the so called \textit{layer problem}
\begin{equation}\label{eq-def-fast-system}
    \mathbf{x}' = f(\mathbf{x},\mathbf{y},0), \ \ \
    \mathbf{y}' = 0.
\end{equation}

 Moreover, the system \eqref{eq-def-fast-system} can be seen as a system of ordinary differential equations, where $\mathbf{y}\in\mathbb{R}^{2}$ is a parameter and the critical set $C_{0}$ is a set of equilibrium points of \eqref{eq-def-fast-system}.

The main goal of \textit{geometric singular perturbation theory} is to study systems \eqref{eq-def-slow-system} and \eqref{eq-def-fast-system} in order to obtain information of the full system \eqref{eq-def-slowfast-1}. Observe that the systems \eqref{eq-def-slowfast-1} and \eqref{eq-def-slowfast-2} are equivalent when $\varepsilon > 0$, since they only differ by time scale.

We next introduce the notion of normally hyperbolic points.\\

Let $\mathbf{x}_{0}\in S$, for any set $S\subset\mathcal{M}$. We say that $\mathbf{x}_{0}$ is \textit{normally hyperbolic} if the $2\times 2$ matrix $Df_{\mathbf{x}}(\mathbf{x}_{0})$ does not have eigenvalues with zero real part. 

In what follows we present a version of \textit{Fenichel's Theorem} whose proof can be found in \cite{Szmolyan}.
\begin{theorem}\label{teo:fenichel}
Consider $\mathcal{M}$ a $C^{r+1}$ manifold, $2\leq r\leq \infty.$ Let $X_\varepsilon,$ $\varepsilon\in(-\varepsilon_0,\varepsilon_0)$ be a $C^r$ family of vector fields on $\mathcal{M},$ and assume $C_0$ a $C^r$ submanifold of $\mathcal{M}$ consisting entirely of equilibrium points of $X_0.$ If $S\subset C_0$ is a $j-$dimensional compact normally hyperbolic invariant manifold of the reduced vector field $X_R$ with a $j+j^s-$dimensional local stable manifold $\mathcal{W}^s$ and a $j+j^u-$dimensional local unstable manifold $\mathcal{W}^u,$ then there exists $\varepsilon_1>0$ such that
\begin{itemize}
    \item[(i)] There exists a $C^{r-1}$ family of manifolds $\{S_\varepsilon:\varepsilon\in(-\varepsilon_1,\varepsilon_1)\}$ with $S_0=S$ and $S_\varepsilon$ is a normally hyperbolic invariant manifold of $X_\varepsilon.$
    \item[(ii)] There are $C^{r-1}$ families of $(j+j^s+k^s)-$dimensional and $(j+j^u+k^u)-$dimensional manifolds $\{S_\varepsilon^s:\varepsilon\in(-\varepsilon_1,\varepsilon_1)\}$ and $\{S_\varepsilon^u:\varepsilon\in(-\varepsilon_1,\varepsilon_1)\}$ such that for $\varepsilon>0$ the manifolds $S_\varepsilon^s$ and $S_\varepsilon^u$ are local stable and unstable manifolds of $S_\varepsilon.$
\end{itemize}
\end{theorem}
Now, suppose that $S$ is given as in Theorem \ref{teo:fenichel}, then $Df_\mathbf{x}(\mathbf{x},\mathbf{y},0)$ is invertible for all $(\mathbf{x},\mathbf{y})\in S.$ From \textit{Implicit Function Theorem} there exists a smooth function $h_0$ defined on a compact domain $K\subset\mathbb{R}^2$ such that
$$S=\{(\mathbf{x},\mathbf{y})):\mathbf{x}=h_0(\mathbf{y})\}.$$
That is, $S$ is locally the graph of a function of $\mathbf{x}$ in terms of $\mathbf{y}$
\begin{theorem}
Assume the same hypotheses of Theorem \ref{teo:fenichel}. If $\varepsilon>0$ is sufficiently small, there exists a function $\mathbf{x}=h_\varepsilon(\mathbf{y})$, defined on $K$, so that the graph
$$S_\varepsilon=\{(\mathbf{x},\mathbf{y})):\mathbf{x}=h_\varepsilon(\mathbf{y})\}$$
is locally invariant under \eqref{eq-def-slowfast-2}. Moreover $h_\varepsilon$ is $C^r$, for any $r<\infty$, jointly in $\mathbf{y}$ and $\varepsilon$.
\end{theorem}

\subsection{Two-dimensional holomorphic slow-fast  systems}
Consider the two-dimensional holomorphic slow-fast system  \eqref{main_eq}, where $f(z,w)=u_1+iv_1$ and $g(z,w)=u_2+iv_2$
are holomorphic functions  of the complex variables $(z,w)=(x_1+iy_1,x_2+iy_2)\in D.$ 

Recall that the fast system associated to system \eqref{main_eq} is given by
\begin{equation}\label{main_eq_2}
\begin{aligned}
\left\{\begin{array}{l}
z'=f(z,w)=u_1+iv_1,\\[5pt]
w'=\varepsilon g(z,w)=\varepsilon(u_2+iv_2),
\end{array} \right.
\end{aligned}
\end{equation}
where the apostrophe ' denotes the derivative of the functions $z(t)$ and $w(t)$ with respect to the variable $t,$ with $t=\tau/\varepsilon$. Notice that the critical set associated to system \eqref{main_eq_2} is defined as $C_0=\{(z,w)\in D:f(z,w)=0\}.$ 

In this context the reduced problem is given by
\begin{equation}\label{eq-def-slow-system_complex}
 0 = f(z,w), \\\
    \dot{w} = g(z,w),
\end{equation}
and the layer problem is defined as
\begin{equation}\label{eq-def-fast-system_complex}
       z' = f(z,w), \\\
    w' = 0.
\end{equation}
Observe that the systems \eqref{main_eq} and \eqref{main_eq_2} are equivalent when $\varepsilon > 0$, since they only differ by time scale. In addition, systems \eqref{eq-def-slow-system_complex} and \eqref{eq-def-fast-system_complex} allow us to obtain information about system \eqref{main_eq}.

Now, we introduce the notion of normal hyperbolicity in the complex context. Let $q_0=(z_0,w_0)\in C_0$ be an equilibrium point of the system \eqref{main_eq_2} for $\varepsilon=0$. Since $f,g$ are holomorphic functions, then the Cauchy-Riemann equations hold. Thus, the Jacobian matrix of the system in $\mathbb{R}^4$ associated to \eqref{main_eq_2} at $q_0$ is given by
\begin{equation}\label{mjacobiana11}
    J_\mathbb{R}(f,0)|_{q_0}=\begin{pmatrix}
(u_1)_{x_1} & -(v_1)_{x_1} & (u_1)_{x_2} & -(v_1)_{x_2} \\
(v_1)_{x_1} & (u_1)_{x_1} & (v_1)_{x_2} & (u_1)_{x_2}  \\
0 & 0 & 0 & 0 \\
0 & 0 & 0 & 0
\end{pmatrix},
\end{equation}
this implies that  the eigenvalues of $J_\mathbb{R}(f,0)|_{q_0}$ are $\lambda_\pm=(u_1)_{x_1}\pm i(v_1)_{x_1}.$ Hence, a subset $S\subset C_0$ is called normally hyperbolic provided that $\operatorname{Re}\left(\frac{\partial f}{\partial z}(q)\right)=(u_1)_{x1}(q)\neq 0$ for all $q\in S.$
\subsection{Briot-Bouquet Systems}
Briot-Bouquet systems have been widely used in the research literature (see, for instance, \cite{Hille,Needham1,Needham2}).

A Briot-Bouquet system is of the following form
\begin{equation}\label{eq_BB}
zw'=F(z, w),  \quad F(0,0)=0,
\end{equation}
where $F$ is a holomorphic function and the apostrophe ' denotes the derivative of the function $w(z)$ with respect to the complex variable $z$. 

The one-dimensional theory is well known since the original work of Briot and Bouquet (see \cite{Briot}). Here we will use the following result, which was proved in \cite[Proposition 1.1.1]{Iwasaki}.
\begin{proposition}[Briot-Bouquet criterion]\label{prop_BB}
If $F(z,w)=\lambda w+\mathcal{O}(z,w^2)$ and $\lambda\notin\mathbb{N}$, then \eqref{eq_BB} admits a unique holomorphic solution at $z=0$ satisfying $w(0)=0$. 
\end{proposition}

\section{Persistence of Invariant Manifold}\label{sec:BB}
This section aims to prove that centers, foci and nodes of the reduced problem associated with the holomorphic slow-fast system \eqref{main_eq}, whose Jacobian matrix of $f$ and $g$ is diagonalizable at the origin, are preserved by singular perturbation. 

For that, first we are going to construct one-dimensional invariant complex manifolds via the \textit{Briot-Bouquet Theory}. Thus, consider the complex differential system given by \eqref{main_eq}.
Without loss of generality, assume that $q_\varepsilon=(0,0)$ is an equilibrium point of \eqref{main_eq} in $D$. Suppose that the associated linearized system is such that the eigenvalues of the Jacobian matrix $J_\mathbb{C}(f,g)|_{q_\varepsilon}$ of $f,g$ are given by $\alpha,\beta\in\mathbb{C}$ and $J_\mathbb{C}(f,g)|_{q_\varepsilon}$ is diagonalizable. Hence, we can write \eqref{main_eq} as \eqref{per_sys_eq_nf}
where
\begin{equation}\label{series_fg}
   \tilde{f}(z,w)  = \displaystyle\sum_{n=2}^\infty\sum_{s+l=n}a_{s,l}z^s w^l \quad\text{and}\quad
   \tilde{g}(z,w)  =  \displaystyle\sum_{n=2}^\infty\sum_{s+l=n}b_{s,l}z^s w^l, 
\end{equation}
which convergent in some neighbourhood of $q_\varepsilon$. Here $s,l$ are nonnegative integers and $a_{s,l},b_{s,l}$ are complex constants. 

In this context, the critical set associated with complex system \eqref{per_sys_eq_nf} is given by $C_0=\{(z,w):\alpha z+\tilde{f}(z,w)=0\}.$ Emphasize that if $C_0$ is a normally hyperbolic critical manifold, then $\operatorname{Re}(\alpha)\neq 0.$  

To find an invariant manifold of system \eqref{per_sys_eq_nf}, we consider the following initial value problem
\begin{equation}\label{eq_nf}
    (\alpha z+\tilde{f}(z,w))\frac{dw}{dz}=\varepsilon(\beta w+\tilde{g}(z,w)),\quad w_\varepsilon(0)=0,\quad w'_\varepsilon(0)=0.
\end{equation}
In what follows we present a fundamental result that will be used to prove the main theorem of this section.
\begin{mtheorem}\label{main_teo}
Suppose that $C_0$ is a normally hyperbolic critical manifold associated with complex system \eqref{per_sys_eq_nf} and
consider $\varepsilon>0$ a small enough parameter. 
Then, system \eqref{per_sys_eq_nf} has a unique one-dimensional invariant complex manifold $C_\varepsilon$ passing through the equilibrium point $q_\varepsilon$, which is given by $$C_\varepsilon=\{(z,w):w=h_\varepsilon(z),\quad h_\varepsilon(0)=0,\quad h'_\varepsilon(0)=0\},$$ where $h_\varepsilon$ is holomorphic in $B=B(0,r)$.
\end{mtheorem}
\begin{proof}
    First, we prove that the initial value problem \eqref{eq_nf} has a unique solution $w=h_\varepsilon(z)$, which is holomorphic in $B.$ Indeed, consider the holomorphic transformation $w(z)=z\varphi(z),$ where $z\in B.$ Then, the initial value problem \eqref{eq_nf} can be written as
    \begin{equation}\label{PVI_1}
        (\alpha z+\tilde{f}(z,z\varphi))[\varphi+z\varphi']=\varepsilon(\beta z\varphi+\tilde{g}(z,z\varphi)),\quad \varphi_\varepsilon(0)=0.
    \end{equation}
 Since $C_0$ is a normally hyperbolic critical manifold, then $\operatorname{Re}(\alpha)\neq 0.$ Thus, dividing equation \eqref{PVI_1} by $\alpha z,$ we get
    \begin{equation}\label{eq_ivp_2}
        \left(1+\frac{\tilde{f}(z,z\varphi)}{\alpha z}\right)[\varphi+z\varphi']=\varepsilon\left(\frac{\beta\varphi}{\alpha }+\frac{\tilde{g}(z,z\varphi)}{\alpha z}\right).
    \end{equation}
    Now, we define the holomorphic functions $p(z,\varphi):=\frac{\tilde{f}(z,z\varphi)}{\alpha z}$ and $q(z,\varphi):=\frac{\tilde{g}(z,z\varphi)}{\alpha z}.$ Then, from \eqref{series_fg} and $w=z\varphi,$ we obtain 
    $$
   \tilde{p}(z,\varphi)  =  \frac{1}{\alpha}\displaystyle\sum_{n=2}^\infty\sum_{s+l=n}a_{s,l}\varphi^l z^{n-1} \quad \text{and} \quad
   \tilde{q}(z,\varphi)  =  \frac{1}{\alpha}\displaystyle\sum_{n=2}^\infty\sum_{s+l=n}b_{s,l}\varphi^l z^{n-1}, 
$$
which converge in some neighbourhood of $q_\varepsilon$. Rewriting  equation \eqref{eq_ivp_2}, we get
\begin{equation*}\label{eq_ivp_3}
        z\varphi'=-\varphi+\varepsilon\dfrac{\frac{\beta\varphi}{\alpha}+q(z,\varphi)}{1+p(z,\varphi)}:=F_\varepsilon(z,\varphi).
    \end{equation*}
    Thus, putting the linear part of $F_\varepsilon$ in an explicit way, we obtain that
\begin{equation}\label{eq_ivp_3}
        z\varphi'=-\left(1-\frac{\varepsilon\beta}{\alpha}\right)\varphi+\frac{\varepsilon b_{20}}{\alpha}z+\varepsilon Q(z,\varphi),
    \end{equation}
    where 
    $$Q(z,\varphi)=\left(q(z,\varphi)-\frac{b_{20}}{\alpha}z\right)-\frac{q(z,\varphi)p(z,\varphi)}{1+p(z,\varphi)}-\frac{\beta\varphi p(z,\varphi)}{\alpha(1+p(z,\varphi))}.$$ 
    Recall that $F_\varepsilon$ is a holomorphic function in a neighborhood of $q_\varepsilon,$ $F_\varepsilon(0,0)=0$ and $Q(z,\varphi)=\mathcal{O}_2(z,\varphi).$ For $\varepsilon>0$ small enough, we have that  $\frac{\varepsilon\beta}{\alpha}-1\notin\mathbb{N},$ then from \textit{Briot-Bouquet criterion} (see Proposition \ref{prop_BB}) we conclude that \eqref{eq_ivp_3} has a unique solution $\varphi=\Lambda_\varepsilon(z)$, which is holomorphic in $B.$ Using that $w=z\varphi,$ then the initial value problem \eqref{eq_nf} has a unique solution $w=h_\varepsilon(z)$, where $h_\varepsilon(z)=z \Lambda_\varepsilon(z)$ is holomorphic in  $B.$

Now, consider the holomoprhic function $H_\varepsilon(z,w)=w- h_\varepsilon(z),$ with $H_\varepsilon(0,0)=0$ and $\frac{\partial H_\varepsilon}{\partial z}(0,0)=0.$ Define the set $C_\varepsilon:=\{(z,w):w=h_\varepsilon(z),\quad h_\varepsilon(0)=0,\quad h'_\varepsilon(0)=0\}$. Emphasize that $C_\varepsilon=H_\varepsilon^{-1}(0).$ By construction, we get that $C_\varepsilon$ is an invariant set. 

Notice that the rank of $H$ is the rank of the matrix $1\times 2$, which is given by  
$$J_\mathbb{C}H_\varepsilon(z,w)=\left(\frac{\partial H}{\partial z}(z,w),\frac{\partial H}{\partial w}(z,w)\right)=\left(-\frac{\partial h_\varepsilon}{\partial z}(z),1\right).$$
Thus, $J_\mathbb{C}H$ has maximal rank (rank 1) in each point of $C_\varepsilon.$ From Proposition \ref{submanifolds}, we conclude that $C_\varepsilon$ is a complex manifold of dimension 1.
\end{proof}

\begin{remark}
    From the proof of Theorem \ref{main_teo}, we can deduce that this result still holds when we replace the hypothesis in which $C_0$ is a normally hyperbolic manifold by $\frac{\partial f}{\partial z}(0,0)=\alpha\neq 0,$ which is a weaker hypothesis.
\end{remark}
In the sequel, assume that:
\begin{itemize}
\item[{\bf (H)}]  the eigenvalues of the $J_\mathbb{C}(f,g)|_{q_\varepsilon}$ are $\alpha,\beta\in\mathbb{C}\setminus\{0\}$ and that $\alpha,\beta$ are of the same type, that is, both $\alpha$ and $\beta$ are pure imaginary complex numbers, real numbers, or complex numbers with nonzero real and imaginary parts.
\end{itemize}
 At this point, we can assume that $C_0$ is given as the graph of a function of $z$ in terms of $w$. Indeed, define $\Theta(z,w)=\alpha z+\tilde{f}(z,w)$ and $q_0=(0,0)$. Since $\Theta(q_0)=0$ and $\frac{\partial \Theta}{\partial z}(q_0)=\alpha\neq 0,$ then from \textit{Implicit Function Theorem}, there exist a neighborhood $U_{q_0}$ of $q_0$ and a holomorphic function $z=L(w)$ such that $L(0)=0$ and
$\Theta(L(w),w)=0,$ for all $w\in U_{q_0}.$ Therefore, locally the critical manifold is given by $C_0=\{(z,w):z=L(w),\quad L(0)=0,\quad w\in U\}.$

Recall that the dynamics on $C_0$ is given by 
\begin{equation}\label{q0_red}
    \dot{w}=\beta w+g(L(w),w)=:G(w).
\end{equation} 
From item $(b)$ of \cite[Theorem 1.1]{GGJ} we know that $G$ and $\beta w$ are 0-conformally conjugate.

We are ready to state the main result of this section.
\begin{mtheorem}\label{prop_center_2} Suppose that the eigenvalues of $J_\mathbb{C}(f,g)|_{q_\varepsilon}$ satisfy hypothesis {\bf (H)}, where $\varepsilon$ is a small enough positive parameter. Then,
    centers, foci, and nodes of the reduced problem associated to \eqref{main_eq} are persistent by singular perturbation.
\end{mtheorem}
\begin{proof}
    By Theorem \ref{main_teo}, we know that, for $\varepsilon>0$ small enough, system \eqref{per_sys_eq_nf}  has a unique one-dimensional invariant complex manifold $C_\varepsilon$ at $q_\varepsilon$, which is given by $$C_\varepsilon=\{(z,w):w=h_\varepsilon(z),\quad h_\varepsilon(0)=0,\quad h'_\varepsilon(0)=0\},$$ where $h_\varepsilon$ is holomorphic in $B$.

    Emphasize that the dynamics about $C_\varepsilon$ is given by 
    \begin{equation}\label{q0_sf}
        \dot{z}=\frac{\dot{w}}{h'_\varepsilon(z)}=\frac{\alpha}{\varepsilon} z+\frac{\tilde{f}(z,h_\varepsilon(z))}{\varepsilon}=:\Gamma_\varepsilon(z).
    \end{equation} 
    By item $(b)$ of \cite[Theorem 1.1]{GGJ} we can conclude that 
    $\Gamma_\varepsilon$ and $\frac{\alpha}{\varepsilon}z$ are 0-conformally conjugate. From \eqref{q0_red} and \eqref{q0_sf} and the fact that $\alpha$ and $\beta$ are of the same type the result follows.
\end{proof}

\section{Other families of holomorphic slow-fast systems}\label{sec:mainresults}
In this section, we are concerned with finding one-dimensional invariant complex manifolds of slow-fast systems whose Jacobian matrix is not necessarily diagonalizable at the origin. Furthermore, we do not require that the origin be an equilibrium point of the system.
\subsection{Uncoupled differential systems}
Consider the two-dimensional holomorphic slow-fast system \eqref{main_eq},
where $f,g$ are holomorphic functions defined in a punctured disk $D=D(0,R)\setminus\{0\}.$

 We assume that $\frac{\partial f}{\partial w}(z,w)\neq 0,$ for all $(z,w)\in C_0.$ From Proposition \ref{submanifolds}, we know that $C_0$ is a complex manifold. Moreover, we can assume that $C_0$ is given as the graph of a function of $w$ in terms of $z$. Indeed, since $\frac{\partial f}{\partial w}(z,w)\neq 0$ for all $(z,w)\in C_0,$ then from \textit{Implicit Function Theorem}, there exist an open set $U$ and a holomorphic function $w=\lambda(z)$ such that 
$f(z,\lambda(z))=0,$ for all $z\in U.$

Let us introduce the equations for the orbits of system \eqref{main_eq}:
\begin{equation}\label{uncouple_eq_22}
   \frac{dw}{dz}=\varepsilon\frac{g(z,w)}{f(z,w)}.
\end{equation}
Suppose that there exist holomorphic functions $\eta,\kappa$ defined in $\Omega=B\setminus L$ such that $\frac{g(z,w)}{f(z,w)}=\eta(z)\kappa(w)$ for all $z,w\in\Omega$, where $L$ is a ray starting at 0, $B=B(0,r)\setminus\{0\}$ and
\begin{itemize}
    \item either 0 is the single zero of $\eta$ (resp. $\kappa$) in $B$;
    \item or 0 is an isolated singularity of $\eta$ (resp. $\kappa$) and $\eta$ (resp. $\kappa$) has no zeros in $B$.
\end{itemize} 
From \cite[Corollary 6.16]{Conway}, we know that the holomorphic functions $\eta,\kappa: \Omega\rightarrow \mathbb{C}$ have a primitive  in $\Omega.$ Thus $1/\eta,1/\kappa$ are also holomorphic functions in $\Omega$ and consequently both $\eta,\kappa$ and $1/\eta,1/\kappa$  have primitives in $\Omega.$ Integrating the equation $\frac{dw}{dz}=\varepsilon\eta(z)\kappa(w)$, we get that there exist holomorphic functions $F,G$ such that $G(w)=\varepsilon F(z),$ $F'(z)=\eta(z)$ and $G'(w)=1/\kappa(w)$. 

In what follows we state an interesting result about the existence of invariant complex manifolds associated with slow-fast systems whose equation of the orbits is separable.
\begin{mtheorem}\label{uncouple1} Consider the slow-fast system \eqref{main_eq} and assume that there exist holomorphic functions $\eta,\kappa$ defined in $\Omega$ such that $\frac{g(z,w)}{f(z,w)}=\eta(z)\kappa(w),$ for all $z,w\in\Omega.$ Then, there exists a one-dimensional invariant complex manifold $C_\varepsilon$ associated with the differential system \eqref{main_eq}, such that
\begin{itemize}
    \item[(a)] the flow on $C_\varepsilon$ converges to the slow flow as $\varepsilon\to0$.
    \item[(b)] the dynamics over $C_\varepsilon$ is given by $\dot{z}= g(z,w).$
\end{itemize}
\end{mtheorem}
\begin{proof}
Consider the holomoprhic function $H_\varepsilon(z,w)=G(w)-\varepsilon F(z)$ and the set $C_\varepsilon=\{(z,w):G(w)=\varepsilon F(z)\}$. Recall that $C_\varepsilon=H_\varepsilon^{-1}(0).$ From \eqref{uncouple_eq_22}, we obtain that $C_\varepsilon$ is an invariant set. 

Emphasize that the rank of $H_\varepsilon$ is the rank of the matrix $1\times 2$, which is given by  
$$J_\mathbb{C}H_\varepsilon(z,w)=\left(\frac{\partial H_\varepsilon}{\partial z}(z,w),\frac{\partial H_\varepsilon}{\partial w}(z,w)\right)=\left(-\varepsilon\eta(z),\frac{1}{ \kappa(w)}\right).$$
Hence, $J_\mathbb{C}H_\varepsilon$ has maximal rank (rank 1) in each point of $C_\varepsilon.$ By Proposition \ref{submanifolds}, we conclude that $C_\varepsilon$ is a complex manifold of dimension 1.

Now, we shall prove item $(a).$ Deriving the equation $G(w)=\varepsilon F(z),$ we get that $1=\varepsilon k(w)\eta(z),$ which implies that $f(z,w)=\varepsilon g(z,w).$ Therefore, $C_\varepsilon\subset D_\varepsilon=\{(z,w):f(z,w)=\varepsilon g(z,w)\}$ and $D_0=C_0.$

Since $\frac{\partial H_\varepsilon}{\partial w}(z,w)\neq 0$ for all $(z,w)\in C_\varepsilon,$ then from \textit{Implicit Function Theorem}, there exist an open set $V$ and a holomorphic function $w=h_\varepsilon(z)$ such that 
$H_\varepsilon(z,h_\varepsilon(z))=0,$ for all $z\in V.$ Then, expanding $h_\varepsilon(z)$ in Taylor series around $\varepsilon=0,$ we get that $h_\varepsilon(z)=\lambda(z)+\varepsilon Q(z,\varepsilon),$ for all $\varepsilon\in [0,\varepsilon_0]$ and $z\in V\cap U,$ where $Q$ is a holomorphic function and $C_0=\{(z,w):w=\lambda(z),\, z\in U\}$ is locally the critical manifold.  

In what follow, we show that the Hausdorff distance between $C_0$ and $C_{\e}$ is of order $\e$, that is $d_H(C_0,C_{\e})=\mathcal{O}(\e).$ Recall that the Hausdorff distance $d_H$ between nonempty subsets $A$ and $B$ is given by  
\[
d_H(A,B)=\max\left\{\sup_{x\in A}\inf_{y\in B}d(x,y)\,,\,\sup_{x\in B}\inf_{y\in A}d(x,y)\right\}.
\]
Hence, $d_H(C_0,C_{\e})=\mathcal{O}(\e)$ provided that the following statements hold:
\begin{itemize}
\item[i.] for each $p\in C_{\e}$ there exists $q\in C_0$ such that $d(p,q)=\mathcal{O}(\e);$
\item[ii.] for each $q\in C_0$ there exists $p\in C_{\e}$ such that $d(p,q)=\mathcal{O}(\e).$
\end{itemize}

In the sequel, we shall verify item $(i)$. Item $(ii)$ can be verified analogously.

Consider the point $p=(z,w)\in C_\varepsilon$ and take $q=(z,\xi)\in C_0.$ Then, $w=h_\varepsilon(z)=\lambda(z)+\varepsilon Q(z,\varepsilon)$  and $\xi=\lambda(z).$ Thus, 
$$d(p,q)=||(0,w-\xi)||=\left|\left|\left(0,\varepsilon Q(z,\varepsilon)\right)\right|\right|=\left|Q(z,\varepsilon)\right|\varepsilon\leq\varepsilon M,$$
where we have used that $|Q(z,\varepsilon)|\leq M,$ for all $(z,\varepsilon)\in K_{V\cap U}\times [0,\varepsilon_0]$, with $K_{V\cap U}$ a compact subset of $V\cap U$. This implies that $d(p,q)=\mathcal{O}(\varepsilon).$

Finally, we prove item $(b)$. Since $H_\varepsilon(z,w)=0$ for all $(z,w)\in C_\varepsilon,$ then $\frac{\partial H_\varepsilon}{\partial z}\dot{z}+\frac{\partial H_\varepsilon}{\partial w}\dot{w}=0.$ Consequently, the dynamics over $C_\varepsilon$ is determined by
\[
\dot{w}= -\dfrac{\dfrac{\partial H_\varepsilon}{\partial z}\dot{z}}{\dfrac{\partial H_\varepsilon}{\partial w}}= -\dfrac{-\varepsilon\eta(z)\dot{z}}{\dfrac{1}{\kappa(w)}}=\varepsilon \underbrace{\eta(z)\kappa(w)}\dot{z}=\dfrac{ g(z,w)}{f(z,w)}f(z,w)= g(z,w).
\]
\end{proof}
Emphasize that it is possible to adapt the previous result for uncoupled slow-fast systems of the form:
\begin{equation}\label{uncouple_eq_1}
\begin{aligned}
\left\{\begin{array}{l}
\varepsilon\dot{z}=f(z),\\[5pt]
\dot{w}=g(w),
\end{array} \right.
\end{aligned}
\end{equation}
where $f, g$ are holomorphic functions defined in $\Omega.$ Thus, the following result is a direct consequence of the above.

\begin{corollary}\label{teo:eq_pole}
    Equilibrium points and poles of order $n$ of the reduced problem associated to \eqref{uncouple_eq_1} are persistent by singular perturbation.
\end{corollary}


\begin{example}
Let $n$ be natural number with $n\geq 2.$ Consider the following system
\begin{equation}\label{uncouple_eq_3}
\begin{aligned}
\left\{\begin{array}{l}
\varepsilon\dot{z}=z^n,\\[5pt]
\dot{w}=g(w),
\end{array} \right.
\end{aligned}
\end{equation}
where $g$ is one of the normal forms given in \cite[Theorem 1.1]{GGJ}. Recall that $C_0=\{(z,w)\in \mathbb{C}^2:z=0\}$ is the critical manifold associated with system \eqref{uncouple_eq_3}, which is not normally hyperbolic. From Theorem \ref{uncouple1}, there exists an invariant complex manifold $C_\varepsilon$ given by:
\begin{itemize}
    \item $C_\varepsilon=\{(z,w):z^{n-1}\ln(w)=\varepsilon\frac{\eta}{-n+1}\},$ provided that $g(w)=\eta w,$ with $\eta\in\mathbb{C}.$
    \item $C_\varepsilon=\{(z,w):z^{n-1}=\varepsilon\frac{m-1}{n-1}w^{m-1}\},$ provided that $g(w)=w^m,$ with $m\geq 2.$
    \item $C_\varepsilon=\{(z,w):z^{n-1}=\varepsilon\frac{m+1}{-n+1}w^{-(m+1)}\},$ provided that $g(w)=\frac{1}{w^m},$ with $m\geq 2.$
    \item $C_\varepsilon=\left\{(z,w):z^{n-1}\left(\frac{w^{-m+1}}{-m+1}+\ln(w)\right)=\varepsilon\frac{\gamma}{-n+1}\right\},$ provided that $g(w)=\frac{\gamma w^m}{1+w^{m-1}},$ with $m\geq 2$ and $\gamma\in\mathbb{C}.$
\end{itemize}
 Moreover, in either case we have that $C_\varepsilon$ converges to $C_0$ when $\varepsilon$ tends to 0.
\end{example}
\subsection{Coupled differential equations}
Consider the system \eqref{couple_eq_1}, 
where $g, G$ are holomorphic functions defined in $\Omega$ such that $G'(w)=1/g(w),$ $\alpha,\beta\in\mathbb{C}\setminus\{0\}.$ In this constext, the critical manifold of \eqref{couple_eq_1} is $C_0=\{(z,w):z=-\frac{\beta G(w)}{\alpha}\}$.

Let us introduce the equations for the orbits of system \eqref{couple_eq_1}:
\begin{equation}\label{couple_eq_22}
    \varepsilon\frac{dz}{dw}=\frac{\alpha z+\beta G(w)}{g(w)}.
\end{equation}

Since $\dot{w}=g(w),$ then $G(w)=t.$ From the first equation of \eqref{couple_eq_1}, we get the linear differential equation given by $\varepsilon\dot{z}=\alpha z+\beta t.$ Recall that $z=-\frac{\beta}{\alpha^2}(\alpha t+\varepsilon)$ is a solution of previous equation. Thus, we obtain the equation $z=-\frac{\beta}{\alpha^2}(\alpha G(w)+\varepsilon)$, which satisfies the equation of orbits \eqref{couple_eq_22}. 

In the next result we prove that the set of points that satisfies the above equation is in fact a one-dimensional invariant complex manifold. Moreover, if we assume that the critical manifold is not normally hyperbolic, then the \textit{geometric singular perturbation theory} does not determine the type of stability of the invariant manifold $C_\varepsilon$. However, in the following theorem  we are giving conditions to ensure when the invariant manifold $C_\varepsilon$ is exponentially attractive (see Figure \ref{fig-exponential}).

\begin{figure}[h!]
  		\begin{overpic}[scale=0.5]{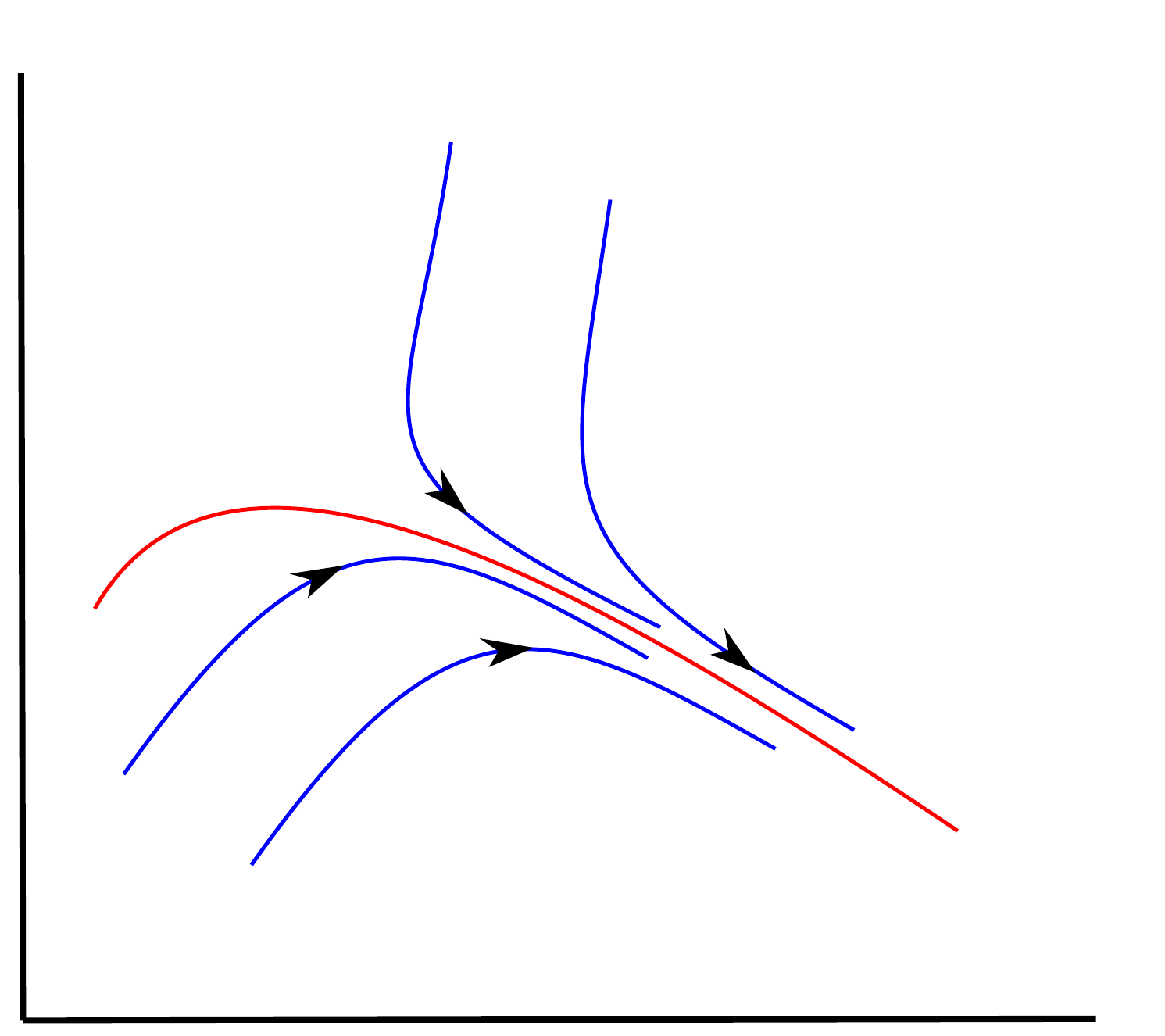}
		\put(96,0){$\mathbb{C}$}
		\put(0,86){$\mathbb{C}$}
		\put(84,15){$C_\varepsilon$}
		\end{overpic}
  \caption{\footnotesize{$C_\varepsilon$ is exponentially attractive in a certain region $\mathcal{R}$, when $C_0$ is not normally hyperbolic.}}
  \label{fig-exponential}
\end{figure}

\begin{mtheorem}\label{couple} Consider the slow-fast system \eqref{couple_eq_1}. Then, there exists a one-dimensional invariant complex manifold $C_\varepsilon$ associated with the differential system \eqref{couple_eq_1}, such that
\begin{itemize}
    \item[(a)] The flow on $C_\varepsilon$ converges to the slow flow as $\varepsilon\to0$.
    \item[(b)] The dynamics over $C_\varepsilon$ is given by $\dot{w}=g(w).$
    \item[(c)] If $\alpha=i\alpha_2,$ with $\alpha_2\in\mathbb{R}\setminus\{0\},$ then the orbit of system \eqref{couple_eq_22} with initial condition $z(w_0) = z_0$ stays exponentially close to the invariant manifold $C_\varepsilon$ in the region $\mathcal{R}=\{(z,w):\alpha_2\Im(G(w)-G(w_0))\geq\eta>0\}$, for some $\eta>0$. 
\end{itemize}
\end{mtheorem}
\begin{proof}
Consider the holomoprhic function $H_\varepsilon(z,w)=z+\frac{\beta}{\alpha^2}(\alpha G(w)+\varepsilon)$ and the set $C_\varepsilon=\{(z,w):z=-\frac{\beta}{\alpha^2}(\alpha G(w)+\varepsilon)\}$. Recall that $C_\varepsilon=H_\varepsilon^{-1}(0).$ By construction $C_\varepsilon$ is an invariant set. 

Emphasize that the rank of $H_\varepsilon$ is the rank of the matrix $1\times 2$, which is given by  
$$J_\mathbb{C}H_\varepsilon(z,w)=\left(\frac{\partial H_\varepsilon}{\partial z}(z,w),\frac{\partial H_\varepsilon}{\partial w}(z,w)\right)=\left(1,\frac{\beta}{ \alpha g(w)}\right).$$
Hence, $J_\mathbb{C}H_\varepsilon$ has maximal rank (rank 1) in each point of $C_\varepsilon.$ By Proposition \ref{submanifolds}, we conclude that $C_\varepsilon$ is a complex manifold of dimension 1.

Now, we shall prove item $(a).$ For that, it is enough to show that the Hausdorff distance between $C_0$ and $C_{\e}$ is of order $\e$ or equivalently that the following statements hold:
\begin{itemize}
\item[i.] for each $p\in C_{\e}$ there exists $q\in C_0$ such that $d(p,q)=\mathcal{O}(\e);$
\item[ii.] for each $q\in C_0$ there exists $p\in C_{\e}$ such that $d(p,q)=\mathcal{O}(\e).$
\end{itemize}

In the sequel, we shall verify item $(i)$. Item $(ii)$ can be verified analogously.

Consider the point $p=(z,w)\in C_\varepsilon$ and take $q=(\xi,w)\in C_0.$ Then, $z=-\frac{\beta}{\alpha^2}(\alpha G(w)+\varepsilon)$ and $\xi=-\frac{\beta}{\alpha}G(w).$ Thus, 
$$d(p,q)=||(z-\xi,0)||=\left|\left|\left(-\frac{\beta}{\alpha^2}\varepsilon,0\right)\right|\right|=\left|\frac{\beta}{\alpha^2}\right|\varepsilon.$$
This implies that $d(p,q)=\mathcal{O}(\varepsilon).$


Next, we prove item $(b)$. Since $H_\varepsilon(z,w)=0$ for all $(z,w)\in C_\varepsilon,$ then $\frac{\partial H_\varepsilon}{\partial z}\dot{z}+\frac{\partial H_\varepsilon}{\partial w}\dot{w}=0.$ Consequently, the dynamics over $C_\varepsilon$ is determined by
\[
\dot{w}= -\dfrac{\dfrac{\partial H_\varepsilon}{\partial z}\dot{z}}{\dfrac{\partial H_\varepsilon}{\partial w}}= -\dfrac{\dot{z}}{\dfrac{\beta}{\alpha g(w)}}=-\dfrac{\alpha z+\beta G(w)}{\dfrac{\varepsilon\beta}{\alpha g(w)}}=-\dfrac{-\dfrac{\varepsilon\beta}{\alpha}}{\dfrac{\varepsilon\beta}{\alpha g(w)}}=g(w).
\]

Finally, we prove item $(c).$ Define $h_\varepsilon(w)=-\frac{\beta}{\alpha^2}(\alpha G(w)+\varepsilon).$ Hence, the invariant manifold is given by $C_\varepsilon=\{(z,w):z=h_\varepsilon(w)\}.$

Now, we perform the change of variables $v = z-h_\varepsilon(w)$ in equation \eqref{couple_eq_22} obtaining:
\begin{equation}\label{new_eq}
    \varepsilon\frac{dv}{dw}=\frac{i\alpha_2}{g(w)}v.
\end{equation}

Recall that the solution of \eqref{new_eq} with initial condition $v(w_0) = z(w_0)-h_\varepsilon(w_0)$ can be written as
$v(w)=v(w_0)e^{\frac{i\alpha_2 (G(w)-G(w_0))}{\varepsilon}}.$ Thus, 
$$|v(w)|=|v(w_0)|e^{-\dfrac{\alpha_2\Im(G(w)-G(w_0))}{\varepsilon}}\leq |v(w_0)|e^{-\dfrac{\eta}{\varepsilon}}.$$
Since $\eta>0,$ then any solution gets exponentially closer to the invariant manifold $C_\varepsilon.$ 
\end{proof}
\begin{example}
    Consider the system defined as 
\begin{equation}\label{ex_eq_1}
\begin{aligned}
\left\{\begin{array}{l}
\varepsilon\dot{z}=iz+w^2,\\[5pt]
\dot{w}=\frac{1}{2w},
\end{array} \right.
\end{aligned}
\end{equation}
with $(z_0,w_0)=(\varepsilon+i,1)$. Notice that $G(w)=w^2$ and $\alpha=i.$ Since $G'(w)=2w,$ from Theorem \ref{couple} there exists a one-dimensional invariant complex manifold $C_\varepsilon$ associated with the differential system \eqref{ex_eq_1}, which is exponentially attractive on $\mathcal{R}=\{(z,w)\in\mathbb{C}^2:\operatorname{Re}(w)\operatorname{Im}(w)>0\}.$
\end{example}

\section{Fenichel manifold approximation}\label{sec:Fenichelaprox}
This section is focused on using Laurent series and \textit{Fenichel's Theorem} to approximate complex manifolds and find smooth manifolds without complex structure of the slow-fast system \eqref{main_eq_aprox}. In particular, we are interested in studying the dynamics of the system \eqref{main_eq_aprox}. 
\subsection{Non-existence of smooth manifolds with complex structure}
Below we present 2 families of slow-fast systems that, despite having associated smooth locally invariant manifolds, have no complex structure.

\subsubsection{Linear-$w^n$ case}\label{sec:linear_zn} Consider the following system
\begin{equation}\label{linear_eq_3}
\begin{aligned}
\left\{\begin{array}{l}
\varepsilon\dot{z}=\alpha z+\beta w,\\[5pt]
\dot{w}=w^n,
\end{array} \right.
\end{aligned}
\end{equation}
where $\alpha,\beta\in\mathbb{C}\setminus\{0\}$ and $n\geq 2$. 
Recall that the critical manifold $C_0$ associated with system \eqref{linear_eq_3} is normally hyperbolic provided that $\operatorname{Re}{(\alpha)}=\alpha_1\neq 0$. Let $S_0$ be a compact subset of $C_0.$

We now apply the \textit{Fenichel's Theorem} to the equivalent $\mathbb{C}^\infty$ four-dimensional real system associated with the system \eqref{linear_eq_3}, then there exists a smooth locally invariant manifold $S_\varepsilon=\{(z,w):z=h_\varepsilon(w)\}$ of the slow-fast system \eqref{linear_eq_3}, which is diffeomorphic to $S_0.$

Emphasize that this does not guarantee that the complex function $h_\varepsilon(z)$ is a holomorphic function. Indeed, suppose that there exists an invariant manifold $C_\varepsilon=\{(z,w):z=h_\varepsilon(w)\},$ where 
\begin{equation}\label{ref_laurent_1}
    h_\varepsilon(w)=a_0+\dots+a_nw^n+\dots+a_{2n-1}w^{2n-1}+\mathcal{O}(w^{2n}).
\end{equation} Since $\dot{z}=h_\varepsilon'(w)\dot{w}$ then we obtain the following equation
\begin{equation}\label{ref_eq_1}
    \alpha h_\varepsilon(w)+\beta w=h_\varepsilon'(w)\varepsilon w^n.
\end{equation}
Substituting \eqref{ref_laurent_1} in \eqref{ref_eq_1}, we get that $a_i=0,$ for all $i\neq kn-(k-1),$ $a_1=-\frac{\beta}{\alpha},$ $a_n=-\frac{\beta\varepsilon}{\alpha^2},$ and $a_{kn-(k-1)}=-\frac{\beta^k\varepsilon^k\prod_{j=2}^k [(j-1)n-(j-2)]}{\alpha^{k+1}},$ with $k\geq 2.$
Therefore, $$C_\varepsilon=\left\{(z,w):z=-\frac{\beta}{\alpha}w-\frac{\beta\varepsilon}{\alpha^2}w^n+\sum_{k=2}^\infty a_{kn-(k-1)}w^{kn-(k-1)}\right\}.$$ 
Notice that the series $\sum_{k=2}^\infty a_{kn-(k-1)}w^{kn-(k-1)}$ diverges for $w\neq 0$, which contradicts the existence of the holomorphic function $h_\varepsilon.$


\subsubsection{Linear-$\frac{\gamma w^n}{1+w^{n-1}}$ case}\label{sec:linear_znn} Consider the following system
\begin{equation}\label{linear_eq_5}
\begin{aligned}
\left\{\begin{array}{l}
\varepsilon\dot{z}=\alpha z+\beta w,\\[5pt]
\dot{w}=\frac{\gamma w^n}{1+w^{n-1}},
\end{array} \right.
\end{aligned}
\end{equation}
where $\alpha,\beta\in\mathbb{C}\setminus\{0\}$ and $n\geq 2$. Recall that the critical manifold $C_0$ associated with system \eqref{linear_eq_5} is normally hyperbolic provided that $\operatorname{Re}{(\alpha)}=\alpha_1\neq 0$. Take $S_0$ a compact subset of $C_0.$

We now apply the \textit{Fenichel's Theorem} to the equivalent $\mathbb{C}^\infty$ four-dimensional real system associated with the system \eqref{linear_eq_5}, then there exists a smooth locally invariant manifold $S_\varepsilon=\{(z,w):z=h_\varepsilon(w)\}$ of the slow-fast system \eqref{linear_eq_5}, which is diffeomorphic to $S_0.$

Emphasize that this does not guarantee that the complex function $h_\varepsilon(z)$ is a holomorphic function. Indeed, suppose that  there exists an invariant manifold $C_\varepsilon=\{(z,w):z=h_\varepsilon(w)\},$ where 
\begin{equation}\label{ref_laurent_2}
    h_\varepsilon(w)=a_0+\dots+a_nw^n+\dots+a_{2n-1}w^{2n-1}+\mathcal{O}(w^{2n}).
\end{equation}
Since $\dot{z}=h_\varepsilon'(w)\dot{w},$ then 
\begin{equation}\label{ref_eq_2}
    \alpha h_\varepsilon(w)+\beta w=h_\varepsilon'(w)\varepsilon \frac{\gamma w^n}{1+w^{n-1}}.
\end{equation}
Substituting \eqref{ref_laurent_2} in \eqref{ref_eq_2}, we get that $a_i=0,$ for all $i\neq kn-(k-1),$ $a_1=-\frac{\beta}{\alpha},$ $a_n=-\frac{\beta\gamma\varepsilon}{\alpha^2},$ and $a_{kn-(k-1)}=-\frac{\beta\gamma\varepsilon\prod_{j=2}^k [\alpha-((j-1)n-(j-2))\gamma\varepsilon]}{\alpha^{k+1}},$ with $k\geq 2.$
Therefore, $$C_\varepsilon=\left\{(z,w):z=-\frac{\beta}{\alpha}w-\frac{\beta\gamma\varepsilon}{\alpha^2}w^n+\sum_{k=2}^\infty a_{kn-(k-1)}w^{kn-(k-1)}\right\}.$$ 
Notice that the series $\sum_{k=2}^\infty a_{kn-(k-1)}w^{kn-(k-1)}$ diverges for $w\neq 0$, which contradicts the existence of the holomorphic function $h_\varepsilon.$

The following result is a direct consequence of \cite[Theorem 1.1]{GGJ} and of cases \ref{sec:linear_zn} and \ref{sec:linear_znn}.
\begin{proposition}
    Consider system \eqref{main_eq} such that $g$ has a zero of order $n>1$ and $f$ is a linear holomorphic function, then system \eqref{main_eq} has no complex manifolds of form $C_\varepsilon=\{(z,w):z=h_\varepsilon(w)\}$.
\end{proposition}
\subsection{Approximation of complex invariant manifold}
Here we present 2 families of slow-fast systems that, in the case of having associated locally invariant smooth manifolds with complex structure, can be approximated via Laurent series and also the singularity associated with the reduced problem is preserved by singular perturbation.
\subsubsection{Linear-linear case}\label{sec:linear_linear}
Consider the following system
\begin{equation}\label{linear_eq_2}
\begin{aligned}
\left\{\begin{array}{l}
\varepsilon\dot{z}=az+bw,\\[5pt]
\dot{w}=cz+dw,
\end{array} \right.
\end{aligned}
\end{equation}
where $a\in\mathbb{C}\setminus\{0\}$ and $b,c,d\in\mathbb{C}$. 
Assume that the critical manifold $C_0$ associated with system \eqref{linear_eq_2} is normally hyperbolic, that is, $\operatorname{Re}{(a)}\neq 0$.

Notice that the reduced problem associated with system \eqref{linear_eq_2} is given by
\begin{equation}\label{linear_eq_2_0}
\begin{aligned}
\left\{\begin{array}{l}
0=az+bw,\\[5pt]
\dot{w}=\left(\frac{ad-bc}{a}\right)w=:\widetilde{g}(w).
\end{array} \right.
\end{aligned}
\end{equation}
Thus, the equilibrium point is $w_0=0$ and $\widetilde{g}'(w_0)=\frac{ad-bc}{a}=\alpha+i\beta,$
where $\alpha=\operatorname{Re}\left(\frac{ad-bc}{a}\right)$ and $\beta=\operatorname{Im}\left(\frac{ad-bc}{a}\right).$
In addition, the Jacobian matrix at the equilibrium point of system \eqref{linear_eq_2_0} is
\[J_\mathbb{R}\widetilde{g}|_{w_0}=\begin{pmatrix}
   \alpha  & -\beta \\
   \beta  & \alpha
\end{pmatrix}.
\]
The determinant $D$ and the trace $T$ are
\[D=\alpha^2+\beta^2\quad\text{and}\quad T=2\alpha.\]

\begin{itemize}
    \item $\beta\neq 0.$ We have that if $\alpha>0$ then the origin is a repelling focus of the reduced problem \eqref{linear_eq_2_0}, if $\alpha<0$ then the origin is an attracting focus of \eqref{linear_eq_2_0}.
\item $\beta=0.$ We have that if $\alpha>0$ then the origin is a repelling node of the reduced problem \eqref{linear_eq_2_0} and if $\alpha<0$ then the origin is an attracting node of \eqref{linear_eq_2_0}.
\end{itemize}
Therefore, we get the following table:
	\begin{equation*}\label{table00}
\begin{array}{|| c |c| c | c|c |c||}
\hline
\alpha & j^u & j^s \\
\hline\hline
+	&2 &0  \\
\hline
- &0 &2 \\
\hline
\end{array}
\end{equation*}
where $j^u,j^s$ have been defined in Theorem \ref{teo:fenichel}. In this case, the eigenvalues of the Jacobian matrix \eqref{mjacobiana11} are $\lambda_\pm=\operatorname{Re}{(a)}\pm i\operatorname{Im}{(a)}$ Thus, we have the following table: 
	\begin{equation*}\label{table1}
\begin{array}{|| c |c| c | c|c |c||}
\hline
\operatorname{Re}{(a)} & k^u & k^s \\
\hline\hline
+	&2 &0  \\
\hline
-& 0 &2 \\
\hline
\end{array}
\end{equation*}
where $k^u,k^s$ have been defined in Theorem \ref{teo:fenichel}. Hence, if $\alpha,\operatorname{Re}{(a)}>0$, then $j^u+k^u=4$ and $j^s+k^s=0$. if $\alpha,\operatorname{Re}{(a)}<0$, then $j^u+k^u=0$ and $j^s+k^s=4.$ If $sign(\alpha)\neq sign(\operatorname{Re}{(a)})$, then $j^u+k^u=2$ and $j^s+k^s=2.$ From item $(ii)$ of Theorem \ref{teo:fenichel}:
\begin{itemize}
    \item If $\alpha,\operatorname{Re}{(a)}>0,$ then the equilibrium point $q_\varepsilon=(0,0)$ of system \eqref{linear_eq_2} is a global repelling point. 
    \item If $\alpha,\operatorname{Re}{(a)}<0,$ then the equilibrium point $q_\varepsilon=(0,0)$ of system \eqref{linear_eq_2} is a global attracting point.
    \item If $sign(\alpha)\neq sign(\operatorname{Re}{(a)})$, then the equilibrium point $q_\varepsilon=(0,0)$ of system \eqref{linear_eq_2} is a saddle point.
\end{itemize}
Notice that when $\alpha=0$ we can not use item $(ii)$ of Theorem \ref{teo:fenichel}, however, $C_0$ is normally hyperbolic, thus if we take a compact subset $S_0$ of $C_0$, we can apply the  \textit{Fenichel's Theorem} to the equivalent $\mathbb{C}^\infty$ four-dimensional real system associated with the system \eqref{linear_eq_2}, then there exists a smooth locally invariant manifold $S_\varepsilon=\{(z,w):z=h_\varepsilon(w)\}$ of the slow-fast system \eqref{linear_eq_2}, which is diffeomorphic to $S_0.$

Although this does not guarantee that the complex function $h_\varepsilon(z)$ is a holomorphic function, we would like to know the dynamics on the Fenichel manifold if it exists. Hence, suppose that  $C_\varepsilon=\{(z,w):z=h_\varepsilon(w)\},$ where $h_\varepsilon(w)=\lambda_0+\lambda_1 w+\mathcal{O}(w^2).$ Since $\dot{z}=h_\varepsilon'(w)\dot{w},$ then 
$$ ah_\varepsilon(w)+bw=h_\varepsilon'(w)\varepsilon((\alpha+i\beta)w).$$
This implies that $\lambda_i=0,$ for all $i\neq 1$ and  $\lambda_1=\frac{-b}{a-\varepsilon(\alpha+i\beta)}.$ 
Therefore, $$C_\varepsilon=\left\{(z,w):z=\frac{-b}{a-\varepsilon(\alpha+i\beta)}w\right\}.$$ 
In particular, if $\alpha=0,$ then $$C_\varepsilon|_{\alpha=0}=\left\{(z,w):z=\frac{-b}{a-i\varepsilon\beta}w\right\}.$$
Recall that the dynamics over the manifold $C_\varepsilon|_{\alpha=0}$ is given by the following differential equation
$$\dot{w}=i\beta w.$$

Therefore, if $w_0$ is a center of reduced problem \eqref{linear_eq_2_0}, then $w_\varepsilon$ is a center of system \eqref{linear_eq_2} such that $w_\varepsilon\rightarrow w_0$ when $\varepsilon\rightarrow 0$ (see Figure \ref{fig-linear}).

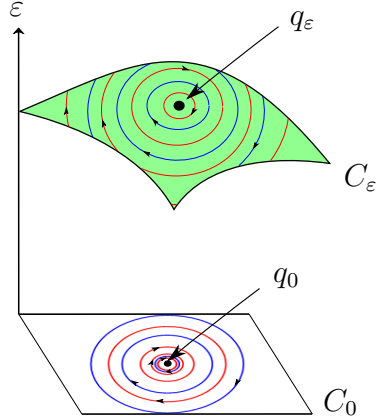
\begin{figure}[h!]
  		\begin{overpic}[scale=0.3]{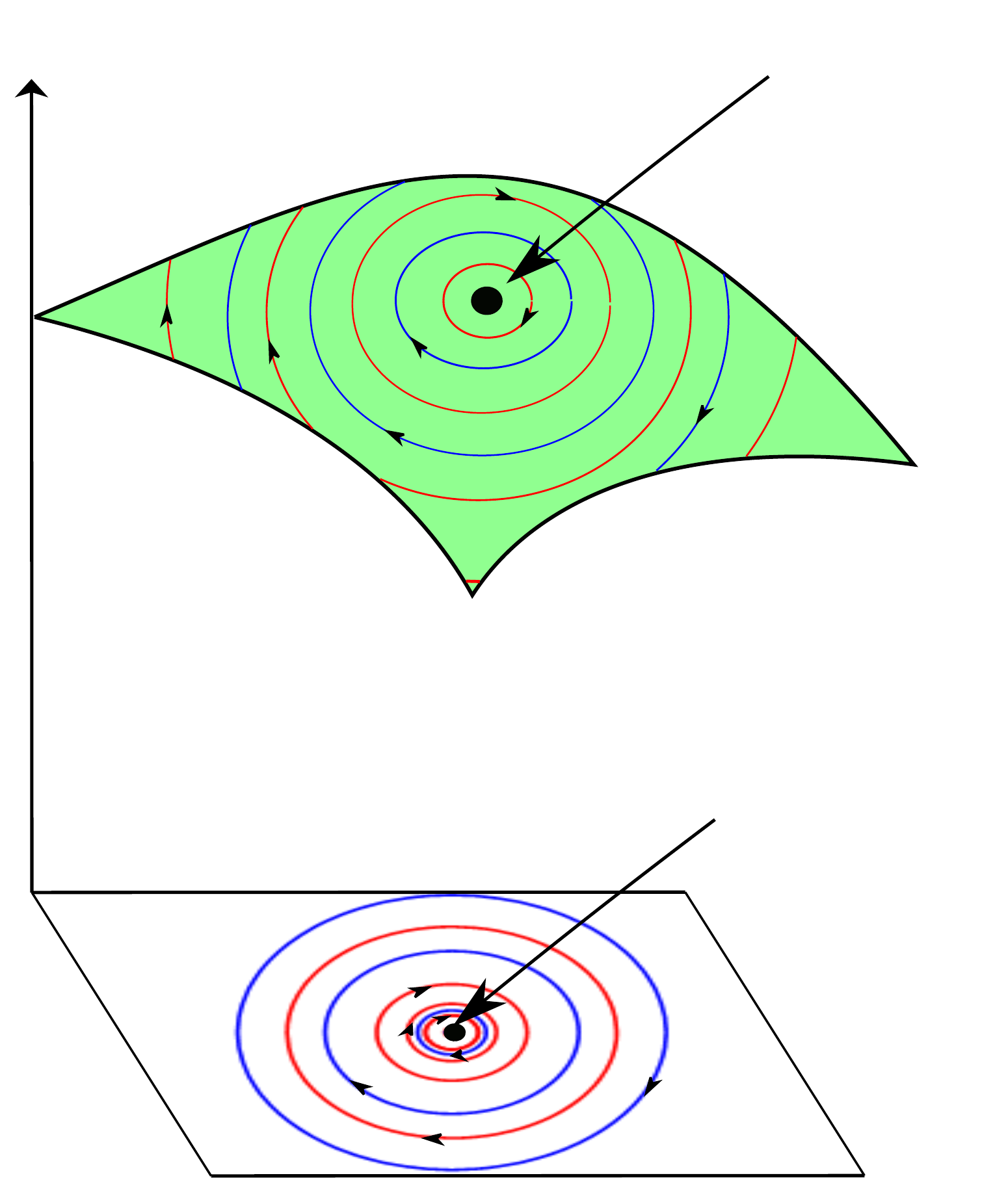}
		\put(76,1){$C_0$}
        \put(81,55){$C_\varepsilon$}
		\put(64,32){$q_0$}
		\put(68,94){$q_\varepsilon$}
        \put(0.5,96){$\varepsilon$}
		\end{overpic}
  \caption{\footnotesize{Persistence of a center of the reduced problem by singular perturbation.}}
  \label{fig-linear}
\end{figure}
The following result is a direct consequence of \cite[Theorem 1.1]{GGJ} and of case \ref{sec:linear_linear}.
\begin{proposition}\label{prop:1}
    Consider system \eqref{main_eq} such that $g$ has a zero $w_0$ of order one and $f$ is a linear holomorphic function. If system \eqref{main_eq} admits an invariant manifold with complex structure, then centers of the reduced problem associated to \eqref{main_eq} are persistent by singular perturbation.
\end{proposition}

\subsubsection{Linear-pole case}\label{sec:linear_pole} Consider the following system
\begin{equation}\label{linear_eq_4}
\begin{aligned}
\left\{\begin{array}{l}
\varepsilon\dot{z}=\alpha z+\beta w,\\[5pt]
\dot{w}=\frac{1}{w^n},
\end{array} \right.
\end{aligned}
\end{equation}
where $\alpha,\beta\in\mathbb{C}\setminus\{0\}$ and $n\geq 1$. 
Recall that the critical manifold $C_0$ associated with system \eqref{linear_eq_4} is normally hyperbolic provided that $\operatorname{Re}{(\alpha)}=\alpha_1\neq 0$. Let $S_0$ be a compact subset of $C_0.$

We now apply the \textit{Fenichel's Theorem} to the equivalent $\mathbb{C}^\infty$ four-dimensional real system associated with the system \eqref{linear_eq_4}, then there exists a smooth locally invariant manifold $S_\varepsilon=\{(z,w):z=h_\varepsilon(w)\}$ of the slow-fast system \eqref{linear_eq_4}, which is diffeomorphic to $S_0.$

Although this does not guarantee that the complex function $h_\varepsilon(z)$ is a holomorphic function, we would like to know the dynamics on the Fenichel manifold if it exists. Hence, suppose that there exists an invariant manifold $C_\varepsilon=\{(z,w):z=h_\varepsilon(w)\},$ where 
\begin{equation}\label{ref_laurent_3}
h_\varepsilon(w)=a_0+\dots+a_{n+2}w^{n+2}+\dots+a_{2n+3}w^{2n+3}+\dots+\mathcal{O}(w^{3n+4}).
\end{equation}
Since $\dot{z}=h_\varepsilon'(w)\dot{w},$ then we get the following equation 
\begin{equation}\label{ref_eq_3}
   \alpha h_\varepsilon(w)+\beta w=h_\varepsilon'(w)\varepsilon w^{-n}. 
\end{equation}
Substituting \eqref{ref_laurent_3} in \eqref{ref_eq_3}, we get that $a_i=0,$ for all $i\neq kn+(k+1)$ and $a_{kn+(k+1)}=\frac{\alpha^{k-1}\beta}{\varepsilon^k\prod_{j=1}^k[jn+(j+1)]},$ with $k\geq 1.$
Therefore, $$C_\varepsilon=\left\{(z,w):z=\sum_{k=1}^\infty a_{kn+(k+1)}w^{kn+(k+1)}\right\}.$$ 
It is easy to see that the series $\sum_{k=1}^\infty a_{kn+(k+1)}w^{kn+(k+1)}$ converges for all $w\in\mathbb{C}$. Recall that the dynamics over the manifold $C_\varepsilon$ is given by the following differential equation
$$\dot{w}=\frac{\beta w+\frac{\alpha\beta}{(n+2)\varepsilon}w^{n+2}+\frac{\alpha^2\beta}{(n+2)(2n+3)\varepsilon^2}w^{2n+3}+\mathcal{O}(1/\varepsilon^3,w^{3n+4})}{\frac{\beta}{\varepsilon}w^{n+1}+\frac{\alpha\beta}{(n+2)\varepsilon^2}w^{2n+2}+\mathcal{O}(1/\varepsilon^3,w^{3n+3})}=:F_\varepsilon(w).$$
Thus, expanding $F_\varepsilon$ around $w_0=0$ we get that
$$F_\varepsilon(w)=\frac{\varepsilon}{w^n}+\mathcal{O}(1/\varepsilon,w^{n+2}).$$
Therefore, $w_0$ is a pole of order $n$ of $F_\varepsilon$ (see Figure \ref{fig-nolinear}).  From \cite[Theorem 1.1]{GGJ}, we conclude that $F_\varepsilon$ and $G\equiv \frac{1}{w^n}$ are 0-conformally conjugated.

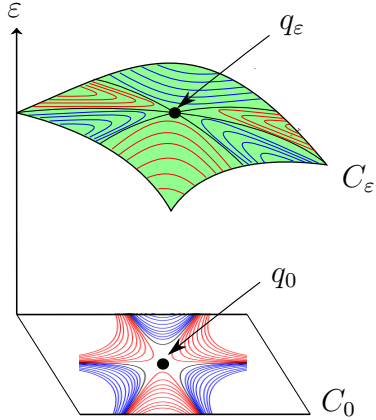
\begin{figure}[h!]
    		\begin{overpic}[scale=0.3]{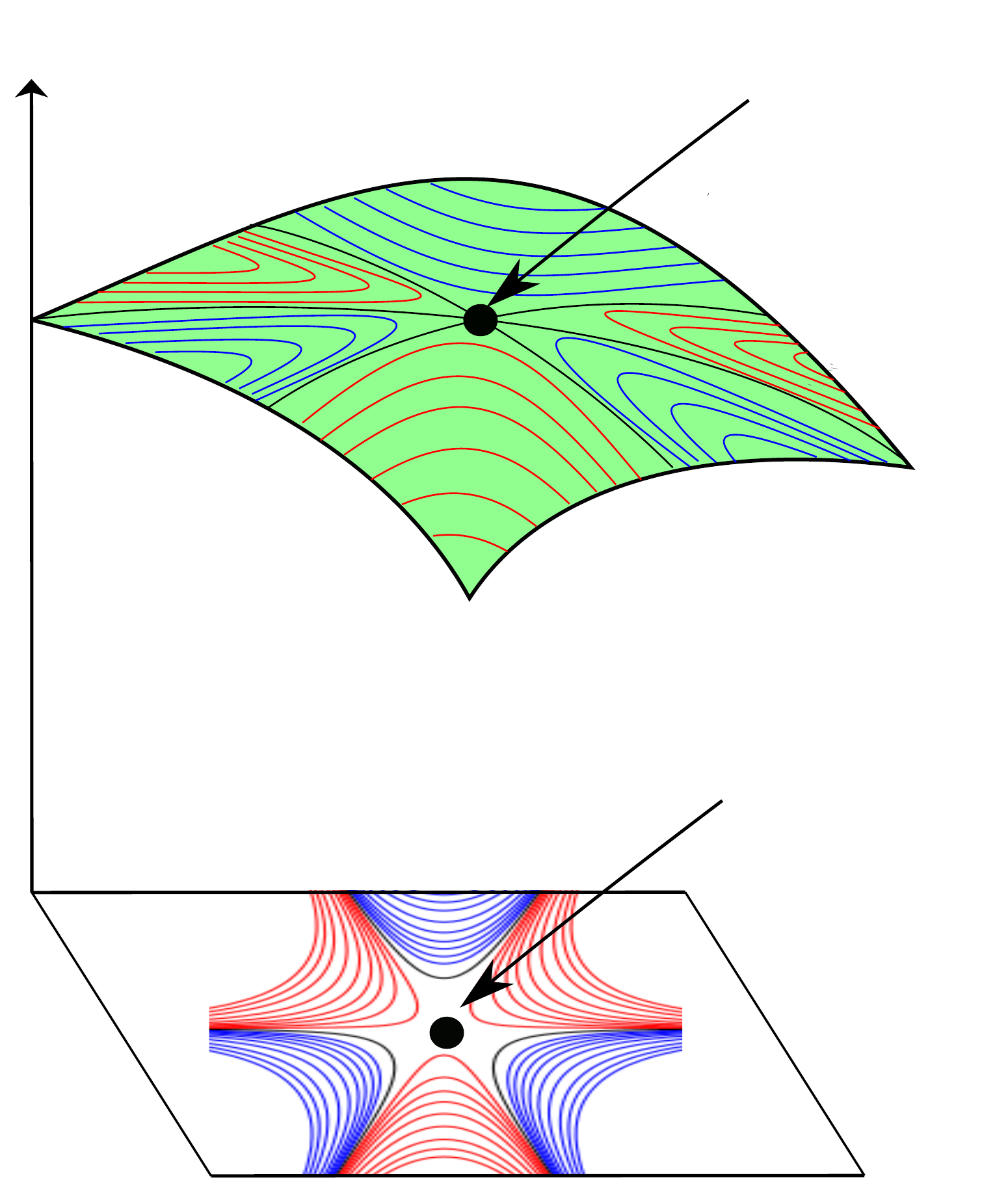}
		\put(76,1){$C_0$}
        \put(81,55){$C_\varepsilon$}
		\put(64,32){$q_0$}
		\put(66,93){$q_\varepsilon$}
        \put(0.5,96){$\varepsilon$}
		\end{overpic}
  \caption{\footnotesize{Persistence of a pole of the reduced problem by singular perturbation.}}
  \label{fig-nolinear}
\end{figure}
The following result is a direct consequence of \cite[Theorem 1.1]{GGJ} and of case \ref{sec:linear_pole}.
\begin{proposition}\label{prop:2}
    Consider system \eqref{main_eq} such that $g$ has a pole $w_0$ of order $n\geq 1$ and $f$ is a linear holomorphic function. If system \eqref{main_eq} admits an invariant manifold with complex structure, then poles of order $n$ of the reduced problem associated to \eqref{main_eq} are persistent by singular perturbation.
\end{proposition}
\section{Statements and Declarations}
The authors declare that they have no known competing financial interests or personal relationships that could have appeared
to influence the work reported in this paper.
\section{Acknowledgements}
This article was possible thanks to the scholarship granted from the Brazilian Federal Agency for Support and Evaluation of Graduate Education (CAPES), in the scope of the Program CAPES-Print, process number 88887.310463/2018-00, International Cooperation Project number 88881.310741/2018-01.  Paulo Ricardo da Silva is also partially supported by São Paulo Research Foundation (FAPESP) grant 2019/10269-3 and CNPq grant 302154/2022-1.

Gabriel Rondón is supported by São Paulo Research Foundation (FAPESP) grants 2020/06708-9 and 2022/12123-9. Luiz Fernando Gouveia is supported by São Paulo Research Foundation (FAPESP) grant 2020/04717-0. 









\bibliographystyle{abbrv}
\bibliography{references1}

\begin{thebibliography}{10}

\bibitem{Briot}
B.~J. Briot, C.
\newblock Recherches sur les propriétés des fonctions définies par des
  équations différentielles.
\newblock {\em J. Écol. Imp. Poly.}, 21:133--197, 1856.

\bibitem{Conway}
J.~B. Conway.
\newblock {\em Functions of one complex variable}, volume~11 of {\em Graduate
  Texts in Mathematics}.
\newblock Springer-Verlag, New York, second edition, 1978.

\bibitem{Fenichel79}
N.~Fenichel.
\newblock Geometric singular perturbation theory for ordinary differential
  equations.
\newblock {\em J. Differential Equations}, 31(1):53--98, 1979.

\bibitem{GGJ}
A.~Garijo, A.~Gasull, and X.~Jarque.
\newblock Normal forms for singularities of one dimensional holomorphic vector
  fields.
\newblock {\em Electron. J. Differential Equations}, pages No. 122, 7, 2004.

\bibitem{Morris}
J.~D. Gray and S.~A. Morris.
\newblock When is a function that satisfies the {C}auchy-{R}iemann equations
  analytic?
\newblock {\em Amer. Math. Monthly}, 85(4):246--256, 1978.

\bibitem{Hille}
E.~Hille.
\newblock {\em Ordinary differential equations in the complex domain}.
\newblock Pure and Applied Mathematics. Wiley-Interscience [John Wiley \&
  Sons], New York-London-Sydney, 1976.

\bibitem{Huybrechts}
D.~Huybrechts.
\newblock {\em Complex geometry}.
\newblock Universitext. Springer-Verlag, Berlin, 2005.
\newblock An introduction.

\bibitem{Iwasaki}
K.~Iwasaki, H.~Kimura, S.~Shimomura, and M.~Yoshida.
\newblock {\em From {G}auss to {P}ainlev\'{e}}.
\newblock Aspects of Mathematics, E16. Friedr. Vieweg \& Sohn, Braunschweig,
  1991.
\newblock A modern theory of special functions.

\bibitem{Jones95}
C.~K. R.~T. Jones.
\newblock Geometric singular perturbation theory.
\newblock In {\em Dynamical systems ({M}ontecatini {T}erme, 1994)}, volume 1609
  of {\em Lecture Notes in Math.}, pages 44--118. Springer, Berlin, 1995.

\bibitem{reg_space}
J.~Llibre, P.~R. da~Silva, and M.~A. Teixeira.
\newblock Regularization of discontinuous vector fields on {$\Bbb R^3$} via
  singular perturbation.
\newblock {\em J. Dynam. Differential Equations}, 19(2):309--331, 2007.

\bibitem{Moroz_b}
I.~M. Moroz.
\newblock Some complex differential equations arising in telecommunications.
\newblock In {\em Singularity theory and its applications, {P}art {II}
  ({C}oventry, 1988/1989)}, volume 1463 of {\em Lecture Notes in Math.}, pages
  278--293. Springer, Berlin, 1991.

\bibitem{Moroz_a}
I.~M. Moroz.
\newblock Bifurcations in a class of complex differential equations.
\newblock {\em European J. Appl. Math.}, 3(3):273--282, 1992.

\bibitem{Needham1}
D.~J. Needham.
\newblock A centre theorem for two-dimensional complex holomorphic systems and
  its generalization.
\newblock {\em Proc. Roy. Soc. London Ser. A}, 450(1939):225--232, 1995.

\bibitem{Needham2}
D.~J. Needham and S.~McAllister.
\newblock Centre families in two-dimensional complex holomorphic dynamical
  systems.
\newblock {\em R. Soc. Lond. Proc. Ser. A Math. Phys. Eng. Sci.},
  454(1976):2267--2278, 1998.

\bibitem{volker}
V.~Scheidemann.
\newblock {\em Introduction to complex analysis in several variables}.
\newblock Birkh\"{a}user Verlag, Basel, 2005.

\bibitem{Szmolyan}
P.~Szmolyan.
\newblock Transversal heteroclinic and homoclinic orbits in singular
  perturbation problems.
\newblock {\em J. Differential Equations}, 92(2):252--281, 1991.

\bibitem{TexeiraSilva12}
M.~A. Teixeira and P.~R. da~Silva.
\newblock Regularization and singular perturbation techniques for non-smooth
  systems.
\newblock {\em Phys. D}, 241(22):1948--1955, 2012.

\bibitem{wellray}
R.~O. Wells, Jr.
\newblock {\em Differential and complex geometry: origins, abstractions and
  embeddings}.
\newblock Springer, Cham, 2017.

\end{thebibliography}

\end{document}